\documentclass[11pt]{article}

\usepackage{fullpage}

\usepackage{amssymb}
\usepackage{amsmath}
\usepackage{amsthm}
\usepackage{bbm}
\usepackage{color}

\usepackage{graphicx} 
\usepackage[all]{xy} 

\newtheorem{theorem}{Theorem}
\newtheorem{lemma}[theorem]{Lemma}

\newtheorem{proposition}[theorem]{Proposition}

\newcommand{\field}[1]{\mathbb{#1}}
\newcommand{\R}{\field{R}}
\newcommand{\N}{\field{N}}
\newcommand{\E}{\field{E}}
\newcommand{\EXP}{\E}
\newcommand{\PROB}{\field{P}}
\renewcommand{\Pr}{\PROB}
\newcommand{\IND}[1]{\mathbbm{1}_{\{#1\}}}

\newcommand{\defeq}{\stackrel{\rm def}{=}}

\newcommand{\C}{\mathcal{C}}
\newcommand{\cP}{\mathcal{P}}

\newcommand{\bN}{\boldsymbol{N}}

\newcommand{\bX}{\boldsymbol{X}}

\newcommand{\by}{\boldsymbol{y}}

\newcommand\inner[2]{ \left\langle {#1},{#2} \right\rangle }
\newcommand{\G}{\field{G}}

\newcommand{\e}{\varepsilon}

\newcommand{\cl}{\mathop{cl}}
\newcommand{\polylog}{\mathop{polylog}}

\newcommand{\BIN}[2]{\mathrm{Bin}\left(#1,#2\right)}

\numberwithin{theorem}{section}

\title{
Exceptional rotations of random graphs: a VC theory
}

\author{
Louigi Addario-Berry
\thanks{Department of Mathematics and Statistics, McGill University; louigi.addario@mcgill.ca}
 \and Shankar Bhamidi
\thanks{Department of Statistics and Operations Research, University Of North Carolina, Chapel Hill; bhamidi@email.unc.edu}
 \and S\'ebastien Bubeck
	\thanks{Microsoft Research and Princeton University; sebubeck@microsoft.com}
 \and Luc Devroye 
	\thanks{School of Computer Science, McGill University; lucdevroye@gmail.com}
\and G\'abor Lugosi 
	\thanks{ICREA and Pompeu Fabra University, Barcelona; gabor.lugosi@upf.edu}
\and Roberto Imbuzeiro Oliveira
\thanks{IMPA, Rio de Janeiro, Brazil; rimfo@impa.br}
}

\begin{document}

\maketitle

\begin{abstract}
In this paper we explore maximal deviations of large random structures from their typical behavior. 
We introduce a model for a high-dimensional random graph process and ask analogous
questions to those of Vapnik and Chervonenkis for deviations of averages: how ``rich''
does the process have to be so that one sees atypical behavior.

In particular, we study a natural process of Erd\H{o}s-R\'enyi random graphs
indexed by unit vectors in $\R^d$. We investigate the deviations of the process with respect to three fundamental properties:  clique number, chromatic number, and connectivity. In all cases we establish upper and lower bounds for the minimal dimension $d$ that guarantees the existence of ``exceptional directions'' in which the random graph behaves atypically with respect to the property.  For each of the three properties, four theorems are established, 
to describe upper and lower bounds for the threshold dimension in the
subcritical and supercritical regimes.
\end{abstract}

\section{Introduction}

One of the principal problems in probability and statistics is the understanding of maximal deviations of averages from their means.  The revolutionary work of Vapnik and Chervonenkis \cite{VaCh71,VaCh74a,VaCh81} introduced a completely new combinatorial approach that opened many paths and helped us understand this fundamental
phenomena. 
Today, the Vapnik-Chervonenkis theory has become the theoretical
basis of statistical machine learning, empirical process theory, and has applications 
in a diverse array of fields. 

The purpose of this paper is to initiate the exploration of maximal deviations of
complex random structures from their typical behavior. We
introduce a model for a high-dimensional random graph process and ask analogous
questions to those of Vapnik and Chervonenkis for deviations of averages: how ``rich''
does the process have to be so that one sees atypical behavior. 
In particular, we study
a process of Erd\H{o}s-R\'enyi random graphs. 
In the $G(n,p)$ model introduced by Erd\H{o}s-R\'enyi \cite{ErRe59,ErRe60}, a graph on $n$ vertices is obtained by 
connecting each pair of vertices with probability $p$, independently, at random. 
The $G(n,p)$ model has been thoroughly studied and many of its properties are
well understood---see, e.g.,  the monographs of Bollob\'as \cite{Bol01} and 
Janson, {\L}uczak, and Ruci\'nski \cite{JaLuRu00}. 

In this paper we introduce a random graph process indexed by unit vectors in $\R^d$, defined as follows.
For positive integer $n$, write $[n]=\{1,\ldots,n\}$.
For $1\le i< j\le n$, let $X_{i,j}$ be independent standard normal vectors in $\R^d$.
Denote by $\bX_n=(X_{i,j})_{1\le i< j\le n}$ the collection of these
random points.
For each $s\in S^{d-1}$ (where $S^{d-1}$ denotes the unit sphere in $\R^d$) and $t\in \R$
we define the random graph $\Gamma(\bX_n,s,t)$ with vertex set $v(\Gamma(\bX_n,s,t)) =[n]$ 
and edge set $e(\Gamma(\bX_n,s,t)) = \{ \{i,j\}: \inner{X_{i,j}}{s}\ge t\}$, 
where $\inner{\cdot}{\cdot}$ denotes the usual 
inner product in $\R^d$. 

For any fixed $s\in S^{d-1}$ and $t\in \R$, $\Gamma(\bX_n,s,t)$
is distributed as an Erd\H{o}s-R\'enyi random graph $G(n,p)$, with 
$p=1-\Phi(t)$ where $\Phi$ is the distribution function of a standard normal
random variable. In particular, $\Gamma(\bX_n,s,0)$ is a $G(n,1/2)$ random
graph. With a slight abuse of notation, we write  $\Gamma(\bX_n,s)$ for
$\Gamma(\bX_n,s,0)$.

We study the \emph{random graph process} 
\[
\G_{d,p}(\bX_n)=   \left\{ \Gamma(\bX_n,s,\Phi^{-1}(1-p)): s\in S^{d-1} \right\}~.
\]
$\G_{d,p}(\bX_n)$ is a stationary process of $G(n,p)$ random graphs, 
indexed by $d$-dimensional unit vectors. For larger values of $d$, the
process becomes ``richer''.
Our aim is to explore how large the dimension $d$ needs to be for there to exist random directions $s$ for which $\Gamma(\bX_n,s,\Phi^{-1}(1-p)) \in \G_{d,p}(\bX_n)$ has different behavior from what is expected from a $G(n,p)$ random graph. Adapting terminology from dynamical percolation \cite{steif2009survey}, we call such directions \emph{exceptional rotations}. More precisely, in analogy with the Vapnik-Chervonenkis theory of studying
atypical deviations of averages from their means, our aim is to develop 
a \emph{VC theory} of random graphs.
In particular, we study three fundamental properties of the graphs in
the family $\G_{d,p}(\bX_n)$: the size of the largest clique, the
chromatic number, and connectivity. In the first two cases we consider
$p=1/2$
while in the study of connectivity we focus on the case when 
$p=c\log n/n$ for some constant $c>0$. 

The graph properties we consider are all monotone, so have a critical probability $p^*$ at which they are typically obtained by $G(n,p)$.  For example, consider connectivity, and suppose we first place ourselves above the critical probability in $G(n,p)$, e.g., $p=c \log n / n$ for $c > 1$, so that $G(n,p)$ is
with high probability connected.  Then the question is how large should $d$
be to ensure that for some member graph in the class, the property (connectivity)
disappears.  There is a threshold {\em dimension} $d$ for this, and we develop
upper and lower bounds for that dimension.
Secondly, consider the regime below the critical probability for connectivity in $G(n,p)$, e.g.,
$p=c \log n / n$ for $c < 1$. In this case, with high probability $G(n,p)$
is not connected, and we ask how large $d$ should be to ensure that 
for some member graph in the class, the property (connectivity)
appears.  Again, we develop upper and lower bounds for the threshold dimension $d$
for this.  

In all, for each of the three properties considered in this paper,
clique number, chromatic number, and connectivity,  four theorems are needed, 
to describe upper and lower bounds for the threshold dimension for exceptional behaviour in the
subcritical regime (when the property typically does not obtain) and in the supercritical regime (when the property typically does obtain). In every case, our results 
reveal a remarkable asymmetry between ``upper'' and ``lower'' deviations relative to this threshold.

Our techniques combine some of the essential notions introduced by Vapnik
and Chervonenkis (such as shattering, covering, packing, and symmetrization), with elements of
high-dimensional random geometry, coupled with sharp estimates for certain random
graph parameters.

The model considered in this paper uses subsets of the collection of halfspaces in $\R^d$ to define
the random graphs in the collection.  
A natural variant would be one in which we associate with each edge $\{i,j\}$ a uniformly
distributed random vector on the torus $[0,1]^d$, and consider a
class parametrized by $s \in [0,1]^d$. Then define the
edge set $e(\Gamma(\bX_n,s,t)) = \{ \{i,j\}: \| X_{i,j} - s \| \le t\}$.
For general
classes of sets of $\R^d$, the complexity of the classes will affect
the behaviour of the collection of random graphs in a universal manner. 
We can define the complexity of a class of graphs indexed in terms of
the threshold dimension needed to make certain graph properties appear or
disappear in the subcritical and supercritical regimes,
respectively.  It will be interesting to explore the relationship
between the combinatorial geometry of the class and these complexities.

Note that when $d=1$, $\G_{1,p}(\bX_n)$ only contains two graphs
(when $p=1/2$, one
is the complement of the other), and therefore the class is trivial.
On the other extreme, when $d\ge \binom{n}{2}$, with probability one,
the collection $\G_{d,1/2}(\bX_n)$ contains \emph{all} $2^{\binom{n}{2}}$ graphs on $n$
vertices. This follows from the following classical result on the ``VC
shatter coefficient'' of linear half spaces
(see, e.g., Schl\"affli \cite{Sch50}, Cover \cite{Cov65})
that determines the number of different graphs in $\G_{d,1/2}(\bX_n)$
(with probability one).

\begin{lemma}
\label{lem:schlaffli}
Given $N\ge d$ points $x_1,\ldots,x_N\in \R^d$ in general position 
(i.e., every subset of $d$ points is linearly independent), 
the number of binary vectors $b\in \{0,1\}^N$
of the form $b= \left(\IND{\inner{x_i}{s} \ge 0}\right)_{i\le N}$
for some $s\in S^{d-1}$ equals
\[
  C(N,d) = 2\sum_{k=0}^{d-1}\binom{N-1}{k}~.
\]
\end{lemma}

In particular, when $N=d$, all $2^{N}$ possible dichotomies of the 
$N$ points are realizable by some linear half space with the origin on its
boundary. In such a case we say that the $N$ points are
\emph{shattered} by half spaces.

\medskip
\noindent{\bf Notation and Overview.} 
Throughout the paper, $\log$ denotes natural logarithm. 
For a sequence $\{A_n\}$ of events, we say that $A_n$ holds \emph{with high probability} if
$\lim_{n\to \infty}\PROB\{A_n\} = 1$.

The paper is organized as follows. In Section \ref{sec:clique} we study the clique
number in the case $p=1/2$. 
The four parts of Theorem \ref{thm:clique} establish upper and lower bounds
for the critical dimension above which, with high probability, 
there exist graphs in $\G_{d,1/2}(\bX_n)$ whose largest clique is
significantly larger/smaller than the typical value, which is $\approx 2\log_2 n-
2\log_2\log_2 n$. We show that the critical dimension for which
some graphs in $\G_{d,1/2}(\bX_n)$ have a clique number at least, say, 
$10 \log_2 n$ is of the order of  $\log^2n /\log\log n$. 

In sharp
contrast to this, $d$ needs to be at least $n^2/\polylog n$ to find 
a graph in $\G_{d,1/2}(\bX_n)$ with maximum clique size $3$ less than the
typical value. We study this functional in Section \ref{sec:chromatic}.
 Theorem \ref{thm:chromatic} summarizes the
four statements corresponding to upper and lower bounds in the sub-,
and super-critical regime. Once again, the two regimes exhibit an
important asymmetry. While no graphs in $\G_{d,1/2}(\bX_n)$ have a chromatic
number a constant factor larger than typical unless $d$ is is of the
order of  $n^2 /\polylog n$, there exist graphs with  a constant
factor smaller chromatic number for $d$ near $n$.

Finally, in Section \ref{sec:connectivity}, connectivity properties
are examined. To this end, we place ourselves in the regime $p=c\log
n/n$ for some constant $c$. When $c<1$, a typical graph $G(n,p)$ is
disconnected, with high probability, while for $c>1$ it is connected. 
In Theorem \ref{thm:connectivity} we address both cases. We show that
for $c>1$, the critical dimension above which one finds disconnected
graphs among $\G_{d,c\log n/n}(\bX_n)$ is of the order of $\log n/\log
\log n$. (Our upper and lower bounds differ by a factor of $2$.)
We also show that when $c<1$, $d$ needs to be at least roughly $n^{1-c}$
in order to find a connected graph $\G_{d,c\log n/n}(\bX_n)$. While we conjecture
this lower bound to be sharp, we do not have a matching upper bound in
this case. However, we are able to show that when $d$ is at least of
the order of $n\sqrt{\log n}$, $\G_{d,c\log n/n}(\bX_n)$ not only
contains some connected graphs but with high probability, 
for any spanning tree, there exists $s\in S^{d-1}$ such that
$\Gamma(\bX_n,s,t)$ contains that spanning tree. This property holds
for even much smaller values of $p$.

In the Appendix we gather some technical estimates required for the proofs.

\section{Clique number}
\label{sec:clique}

In this section we consider $p=1/2$ and investigate the extremes 
of the clique number amongst the graphs $\Gamma(\bX_n,s)$, $s\in S^{d-1}$.
Denote by $\cl(\bX_n,s)$ the size of the largest clique in $\Gamma(\bX_n,s)$.

The typical behavior of the clique number of a $G(n,1/2)$ random graph
is quite accurately described by Matula's classical theorem \cite{Mat72} that states that for any fixed
$s\in S^{d-1}$, for any $\epsilon>0$,
\[
   \cl(\bX_n,s)\in \left\{\lfloor \omega-\epsilon\rfloor, \lfloor \omega+\epsilon\rfloor\right\}
\]
with probability tending to $1$, where
$\omega=2\log_2 n -2\log_2\log_2n +2\log_2 e-1$.

Here we are interested in understanding the values of $d$ for which
graphs with atypical clique number appear.
We prove below that while for moderately large values of $d$ some
graphs $\Gamma(\bX_n,s)$ have a significantly larger clique number
than $\omega$, one does not find graphs with significantly smaller
clique number unless $d$ is nearly quadratic in $n$.

Observe first that by Lemma \ref{lem:schlaffli} for any $k$, if $d\ge
\binom{k}{2}$, then,
with probability one, $\cl(\bX_n,s) \ge k$ for some $s\in S^{d-1}$. 
(Just fix any set of $k$ vertices; all $2^k$ graphs on these vertices
is present for some $s$, including the complete graph.)
For example, when $d\sim (9/2)(\log_2n)^2$, 
$\cl(\bX_n,s) \ge 3\log_2 n$ for some $s\in S^{d-1}$, a quite atypical
behavior.  In fact, with a more careful argument we show below that
when $d$ is a sufficiently large constant multiple of $(\log n)^2/\log
\log n$, then, with high probability, there exists $s\in S^{d-1}$ such that
$\cl(\bX_n,s) \ge 3\log_2 n$. We also show that no such $s$ exists for
$d=o((\log n)^2/\log\log n)$. 
Perhaps more surprisingly, clique numbers significantly smaller than the typical
value only appear for huge values of $d$. The next theorem shows the
surprising fact that 
in order to have that for some $s\in S^{d-1}$, $\cl(\bX_n,s)< \omega
-3$, the dimension needs to be $n^{2-o(1)}$. (Recall that for
$d=\binom{n}{2}$ the point set $\bX_n$ is shattered and one even has
$\cl(\bX_n,s)=1$ for some $s$.
Our findings on the clique number are summarized in the following theorem.

\begin{theorem}
\label{thm:clique}
{\sc (clique number.)}
 If $\cl(\bX_n,s)$ denotes the clique number of
$\Gamma(\bX_n,s)$, then, with high probability the following hold:
\begin{itemize}
\item[(i)] {\sc (subcritical; necessary.)}
If $d=o(n^2/(\log n)^9)$, then for all $s\in S^{d-1}$,
$\cl(\bX_n,s) >  \omega-3$~.
\item[(ii)] {\sc (subcritical; sufficient.)}
If $d\ge \binom{n}{2}$, then there exists $s\in
S^{d-1}$ such that $\cl(\bX_n,s) =1$~.
\item[(iii)] {\sc (supercritical; necessary.)}
For any $c > 2$ there exists $c'>0$ such that if $d \le c' \log^2 n/ \log \log
n$, then for all $s\in S^{d-1}$, we have  
$\cl(\bX_n,s) \le c\log_2 n$. 
\item[(iv)] {\sc (supercritical; sufficient.)}
For any $c>2$ and $c'>c^2/(2\log 2)$, 
if $d\ge c'\log^2 n/\log\log n$, then there exists $s\in
S^{d-1}$ such that $\cl(\bX_n,s) \ge c\log_2 n$~. 
\end{itemize}
The event described in (ii) holds with probability one for all $n$.
\end{theorem}

\begin{proof}
To prove part (i),
let $k=\lfloor \omega -3\rfloor$ and let $N_k(s)$ denote the number of cliques of size $k$ in $\Gamma(\bX_n,s)$.
Let $\eta\in (0,1]$ and let $\C_{\eta}$ be a minimal $\eta$-cover of
$S^{d-1}$. Then
\begin{eqnarray*}
\lefteqn{
\PROB\left\{ \exists s\in S^{d-1}: N_k(s)=0 \right\} } \\
& = &
\PROB\left\{ \exists s'\in \C_{\eta} \ \text{and} \ \exists 
s\in
  S^{d-1}: \|s-s'\|\le \eta: N_k(s)=0 \right\} 
\\
& \le &
|\C_{\eta}| \PROB\left\{ \exists 
s\in
  S^{d-1}: \|s-s_0\| \le \eta: N_k(s)=0  \right\}
\end{eqnarray*}
where $s_0=(1,0,\ldots,0)$ and the last inequality follows from the union bound.
Consider the graph $\Gamma(\bX_n,s_0,-\eta\sqrt{1-\eta^2/2})$ in which
vertex $i$ and vertex $j$ are connected if and only if the
first component of $X_{i,j}$ is at least $-\eta\sqrt{1-\eta^2/2}$

The proof of Lemma \ref{lem:caps_half} implies that the event $\left\{ \exists 
s\in
  S^{d-1}: \|s-s_0\| \le \eta: N_k(s)=0  \right\}$ is included in the
event that $\Gamma(\bX_n,s_0,-\eta\sqrt{1-\eta^2/2})$ does not have
any clique of size $k$.
By Lemma \ref{lem:caps_half}, the probability of this is bounded by the
probability that an Erd\H{o}s-R\'enyi random graph $G(n,1/2-\alpha_n)$
does not have any clique of size $k$ where $\alpha_n= \frac{\eta\sqrt{d}}{\sqrt{2\pi}}$.
If we choose (say) $\eta=1/n^2$ then for $d\le n^2$ we have
$\alpha_n \le 1/n$ and therefore, by Lemma \ref{lem:cliquenum} below,
\[ 
\PROB\left\{ \exists 
s\in
  S^{d-1}: \|s-s_0\| \le \eta: N_k(s)=0  \right\}
\le \exp\left( \frac{-C' n^2}{(\log_2 n)^8}\right)
\]
for some numerical constant $C'$.
Thus, using Lemma \ref{lem:cover},
\[
\PROB\left\{ \exists s\in S^{d-1}: N_k(s)=0 \right\} 
\le (4n^2)^d \exp\left( \frac{-C' n^2}{(\log_2 n)^8}\right) =o(1)
\]
whenever $d=o(n^2/(\log n)^9)$.

\medskip
\noindent
Part (ii) follows from the simple fact that, by Lemma \ref{lem:schlaffli}, with $d=
\binom{n}{2}$ even the empty graph appears among the $\Gamma(\bX_n,s)$.

\medskip
\noindent
The proof of part (iii)
proceeds similarly to that of part (i).
Let $k=c \log_2 n$. Then
\begin{eqnarray*}
\lefteqn{
\PROB\left\{ \exists s\in S^{d-1}: N_k(s)\ge 1 \right\} } \\
& \le &
|\C_{\eta}| \PROB\left\{ \exists 
s\in
  S^{d-1}: \|s-s_0\| \le \eta: N_k(s)\ge 1  \right\}~.
\end{eqnarray*}
Similarly to the argument of (i), we note that
the event $\left\{ \exists 
s\in
  S^{d-1}: \|s-s_0\| \le \eta: N_k(s)\ge 1  \right\}$ is included in the
event that $\Gamma(\bX_n,s_0,-\eta\sqrt{1-\eta^2/2})$ has
a clique of size $k$, which is 
 bounded by the
probability that an Erd\H{o}s-R\'enyi random graph $G(n,1/2+\alpha_n)$
has a clique of size $k$ where $\alpha_n= \frac{\eta\sqrt{d}}{\sqrt{2\pi}}$.
Denoting $p=1/2+\alpha_n$, this probability is bounded by
$\binom{n}{k} p^{\binom{k}{2}} \le \left(n p^{k/2}\right)^k$.
We may choose $\eta=4/d$. Then, for $d$ sufficiently large, 
$\alpha_n\le (c/2-1)\log 2$ and,
 using Lemma \ref{lem:cover}, we have
\begin{eqnarray*}
\PROB\left\{ \exists s\in S^{d-1}: N_k(s)\ge 1 \right\} 
& \le &
(4/\eta)^d \left(n p^{(c/2)\log_2 n}\right)^{c\log_2 n}  \\
&\le &
e^{d\log d} \left(n^{1+(c/2)\log_2(1/2+\alpha_n)}\right)^{c\log_2 n} \\
&\le &
e^{d\log d} \left(n^{1-c/2+c\alpha_n/\log 2}\right)^{c\log_2 n}  \\
&\le &
e^{d\log d} n^{(1-c/2)c(\log_2 n)/2}  \\
&= &
e^{d\log d -(c-2)c(\log_2 n)^2\log 2/4}~,
\end{eqnarray*}
and the statement follows.

\medskip
\noindent
It remains to prove part (iv). The proof relies on the second moment
method.
Let $c>2$, $c'>c^2/(2\log 2)$, and assume
that $d\ge c'\log^2 n/\log\log n$.
Let $K$ be a constant satisfying
$K>2/\sqrt{c'}$ and define $\theta=K\sqrt{\log\log n}/\log n$.
Let $A$ be a subset of $S^{d-1}$ of cardinality at least
$(d/16)\theta^{-(d-1)}$ such that for all distinct pairs $s,s'\in A$,
we have $\inner{s}{s'}\ge \cos(\theta)$. Such a set exists by Lemma
\ref{lem:pack}.
Also, let $\C$ be the family of all subsets of
$[n]$ of cardinality $k=\lfloor c\log_2n\rfloor$. 
For $s\in A$ and $\gamma\in \C$,
denote by $Z_{s,\gamma}$ the indicator that all edges between vertices
in $\gamma$ are present in the graph $\Gamma(\bX_n,s)$. Our aim is to 
show that $\lim_{n\to\infty}\PROB\{Z>0\}= 1$ where
\[
   Z= \sum_{s\in A}\sum_{\gamma\in \C} Z_{s,\gamma}~.
\]
To this end, by the second moment method (see, e.g., \cite{AlSp92}),
it suffices to prove that $\EXP Z \to \infty$ and that $\EXP[Z^2] =
(\EXP Z)^2(1+o(1)$.

To bound $\EXP Z$ note that
\begin{eqnarray*}
\EXP Z & = & |A| \binom{n}{k}\EXP Z_{s,\gamma}  \\
& \ge & (d/16)\theta^{-(d-1)} \binom{n}{k} 2^{-\binom{k}{2}}  \\
& = &
\exp\left((\log n)^2\left(c'-\frac{c^2}{2 \log 2}+\frac{c}{\log
      2}+o(1)\right)\right)
\to \infty~.
\end{eqnarray*}
On the other hand, 
\begin{eqnarray*}
\EXP[Z^2] & = & \sum_{s,s'\in A}\sum_{\gamma,\gamma'\in \C}
\EXP[Z_{s,\gamma}Z_{s',\gamma'}]\\
& = &
\sum_{s,s':s\neq s'\in A}\sum_{\gamma,\gamma'\in   \C:|\gamma\cap\gamma'|\le 1}
\EXP[Z_{s,\gamma}Z_{s',\gamma'}]
+
\sum_{s\in A}\sum_{\gamma,\gamma'\in \C}
\EXP[Z_{s,\gamma}Z_{s,\gamma'}]\\
& & +
\sum_{s,s':s\neq s'\in A}\sum_{\gamma,\gamma'\in
  \C:|\gamma\cap\gamma'|\ge 2}
\EXP[Z_{s,\gamma}Z_{s',\gamma'}] \\
& \defeq & I + II + III~.
\end{eqnarray*}
For the first term note that if $\gamma$ and $\gamma'$ intersect in at
most one vertex then $Z_{s,\gamma}$ and $Z_{s',\gamma'}$ are independent
and therefore 
\[
  I = \sum_{s,s':s\neq s'\in A}\sum_{\gamma,\gamma'\in   \C:|\gamma\cap\gamma'|\le 1}
\EXP Z_{s,\gamma} \EXP Z_{s',\gamma'}
\le (\EXP Z)^2~.
\]
Hence, it suffices to prove that $II+III=o( (\EXP Z)^2)$. To deal with
$II$, we have
\begin{eqnarray*}
\frac{II}{(\EXP Z)^2}
& = &
\frac{1}{|A|\cdot \binom{n}{k}} \sum_{\ell=0}^k 2^{\binom{\ell}{2}}
\binom{n-k}{k-\ell} \binom{k}{\ell} \\
& \le &
\frac{1}{|A|} \sum_{\ell=0}^k 2^{\binom{\ell}{2}}
\frac{k^{2\ell}}{(n-2k)^\ell \ell!} \\
& \le &
\frac{1}{|A|} 2^{\binom{\ell}{2}} \sum_{\ell=0}^\infty
\left(\frac{k^2}{n-2k}\right)^\ell \frac{1}{\ell!} \\
& = &
\exp\left(-(\log n)^2\left(c'+o(1) -c^2/(2\log 2)\right) \right) \to 0~.
\end{eqnarray*}
We now take care of $III$. To this end, we bound
\[
  \max_{\stackrel{s,s'\in A: s\neq s'}{\gamma,\gamma': |\gamma\cap \gamma'|=\ell}}
\EXP[Z_{s,\gamma}Z_{s',\gamma'}] 
\]
by
\[
    2^{\binom{\ell}{2}-2\binom{k}{2}+1} 
\PROB\left\{\inner{\frac{N}{\|N\|}}{s_0} \ge \sin(\theta/2)\right\}^{\binom{\ell}{2}}~,
\]
where $N$ is a standard normal vector in $\R^d$. To see this, note
that $2\binom{k}{2}-\binom{\ell}{2}$ edges of the two cliques occur
independently, each with probability $1/2$. The remaining
$\binom{\ell}{2}$ edges must be in both
$\Gamma(\bX_n,s)$ and $\Gamma(\bX_n,s')$. A moment of thought reveals
that this probability is
bounded by the probability that the angle between a random normal
vector and a fixed unit vector (say $s_0$) is less than
$\pi/2-\theta/2$. This probability may be bounded as
\begin{eqnarray*}
\PROB\left\{\inner{N/\|N\|}{s_0} \ge \sin(\theta/2)\right\}
& = & \frac{1}{2}\PROB\left\{B \ge \sin^2(\theta/2)\right\} \\
& & \text{(where $B$ is a Beta$(1/2,(d-1)/2)$ random variable)} \\
& \le & \frac{\EXP B}{2 \sin^2(\theta/2)}  \\
& = &\frac{1}{2 d\sin^2(\theta/2)}  \\ 
& = &\frac{2+o(1)}{d\theta^2} =  \frac{2+o(1)}{c'K^2}~.
\end{eqnarray*}
Via the same counting argument used in handling $II$, we have 
\begin{eqnarray*}
\frac{III}{(\EXP Z)^2}
\le 
\sum_{\ell=2}^k2^{\binom{\ell}{2}}\left(\frac{2+o(1)}{c'K^2}\right)^{\binom{\ell}{2}}
\left(\frac{k^2}{n-2k}\right)^\ell \frac{1}{\ell!}~.
\end{eqnarray*}
Since $c'K^2>4$, we have, for $n$ large enough,
\begin{eqnarray*}
\frac{III}{(\EXP Z)^2}
\le 
\sum_{\ell=2}^k
\left(\frac{k^2}{n-2k}\right)^\ell \frac{1}{\ell!} =O\left(\frac{(\log
    n)^2}{n^2}\right)
\end{eqnarray*}
as required. This concludes the proof of the theorem.
\end{proof}
We conclude the section by remarking that the above proof extends straightforwardly to $G(n,p)$ for any constant $p \in (0,1)$.

\section{Chromatic number}
\label{sec:chromatic}

A proper coloring of vertices of a graph assigns a color to each
vertex such that no pair of vertices joined by an edge share the same
color.
The \emph{chromatic number} $\chi(G)$ of a graph $G$ is the smallest number of
colors for which a proper coloring of the graph exists. 

Here we study the fluctuations of the chromatic numbers $\chi(\Gamma(\bX_n,s))$
from its typical behavior as $s\in S^{d-1}$. Once again, for
simplicity of the presentation, we consider $p=1/2$. The arguments
extend easily to other (constant) values of $p$.

For a fixed $s$, a celebrated result of Bollob\'as \cite{Bol88}
implies that
\[
\frac{n}{2\log_2n}\le   \chi(\Gamma(\bX_n,s)) 
\le \frac{n}{2\log_2n}(1+o(1)) 
\]
with high probability. 

In this section we derive estimates for the value of the dimension $d$ for
which there exist random graphs in the collection $\G_{d,1/2}(\bX_n)$
whose chromatic number differs substantially (i.e., by a constant
factor) from that of a typical $G(n,1/2)$ graph. Similar to the case
of the clique number studied in Section \ref{sec:clique}, we find that
upper and lower deviations exhibit a different behavior---though in a
less dramatic way.
With high probability, one does not see a graph with a clique number larger than 
$(1+\epsilon)n/(2\log_2n)$ unless $d$ is at least $n^2/\polylog n$.
On the other hand, when $d$ is roughly linear in $n$, there are
graphs is $\G_{d,1/2}(\bX_n)$ with chromatic number at most
$(1-\epsilon)n/(2\log_2n)$. Below we make these statements rigorous
and also show that they are essentially tight. 

\begin{theorem}
\label {thm:chromatic}
{\sc (chromatic number.)}
Let $\epsilon\in (0,1/2)$. If $\chi(\Gamma(\bX_n,s))$ denotes the chromatic number of
$\Gamma(\bX_n,s)$, then, with high probability the following hold:
\begin{itemize}
\item[(i)] {\sc (subcritical; necessary.)}
If $d=o(n/(\log n)^3)$, then for all $s\in S^{d-1}$,
$\chi(\Gamma(\bX_n,s)) \ge (1-\epsilon)n/(2\log_2 n)$.
\item[(ii)] {\sc (subcritical; sufficient.)}
If $d\ge 2n\log_2n /(1-2\epsilon)$, then there exists $s\in S^{d-1}$
such that
$\chi(\Gamma(\bX_n,s)) \le (1-\epsilon)n/(2\log_2 n)$.
\item[(iii)] {\sc (supercritical; necessary.)}
If $d=o(n^2/(\log n)^6)$, then for all $s\in S^{d-1}$,
$\chi(\Gamma(\bX_n,s)) \le (1+\epsilon)n/(2\log_2 n)$.
\item[(iv)] {\sc (supercritical; sufficient.)}
If $d\ge .5\left[(1+\epsilon)n/(2\log_2n)\right]^2$, then there exists $s\in S^{d-1}$
such that
$\chi(\Gamma(\bX_n,s)) \ge (1+\epsilon)n/(2\log_2 n)$.
\end{itemize}
\end{theorem}

Part (i) of Theorem \ref {thm:chromatic} follows from the following
``uniform concentration'' argument. 
\begin{proposition}
If $d=o(n/(\log n)^3)$, we have
\[
   \sup_{s\in S^{d-1}}\left|\chi(\Gamma(\bX_n,s)) 
- \frac{n}{2\log_2n}\right| = o_p\left(\frac{n}{\log_2n}\right)~,
\]
\end{proposition}

\begin{proof}
A classical result of Shamir and Spencer\cite{ShSp87}
shows that for any fixed $s\in S^{d-1}$,
\[
 \left| \chi(\Gamma(\bX_n,s))- \EXP \left(\chi(\Gamma(\bX_n,s))\right)\right| 
  = O_p(n^{1/2})~.
\]
In fact, one may easily combine the above-mentioned results of
Bollob\'as and Shamir and Spencer to obtain that 
\[
   \frac{\EXP \chi(\Gamma(\bX_n,s))}{n/(2\log_2 n)} \to 1~.
\]
The proof of the proposition is based on combining the Shamir-Spencer
concentration argument with Vapnik-Chervonenkis-style symmetrization.

For each $s\in S^{d-1}$ and $i=2,\ldots,n$, define 
$Y_{i,s}=(\IND{\inner{X_{i,j}}{s}\ge 0\}})_{j=1,\ldots,i-1} \in \{0,1\}^{i-1}$ 
as the collection of indicators of edges connecting vertex $i$ smaller-labeled vertices in
$\Gamma(\bX_n,s)$. As Shamir and Spencer, we consider the chromatic
number $\Gamma(\bX_n,s)$  as a function of these variables and define
the function $f:\prod_{i=2}^n \{0,1\}^{i-1}\to \N$ by
\[
f(Y_{2,s},\ldots,Y_{n,s})=  \chi(\Gamma(\bX_n,s))~.
\]
By Markov's inequality, it suffices to show that
\[
   \EXP \left[ \sup_{s\in S^{d-1}} \left|f(Y_{2,s},\ldots,Y_{n,s}) -
       \EXP f(Y_{2,s},\ldots,Y_{n,s}) \right| \right] =
   o\left(\frac{n}{\log n}\right)~.
\]
Let $\bX_n'=(X_{i,j}')_{1\le i< j\le n}$ be an independent copy of $\bX_n$.
Denote by $\EXP'$ conditional
expectation given $\bX_n$.
We write $Y_{i,s}'=(\IND{\inner{X_{i,j}'}{s}\ge 0\}})_{j=1,\ldots,i-1} \in \{0,1\}^{i-1}$.

Also introduce random ``swap operators'' 
$\epsilon_2,\ldots,\epsilon_n$
defined by
\[
  \epsilon_i(Y_{i,s},Y_{i,s}') = \left\{ \begin{array}{ll}
                   Y_{i,s} & \text{with probability $1/2$} \\
                   Y_{i,s}' & \text{with probability $1/2$} 
                    \end{array} \right.
\]
where the $\epsilon_i$ are independent of each other and of everything else.
\begin{eqnarray*}
\lefteqn{
   \EXP \left[ \sup_{s\in S^{d-1}} \left|f(Y_{2,s},\ldots,Y_{n,s}) -
       \EXP f(Y_{2,s},\ldots,Y_{n,s}) \right| \right]
} \\
& = &
  \EXP \left[ \sup_{s\in S^{d-1}} \left|\EXP' \left(f(Y_{2,s},\ldots,Y_{n,s})
      - f(Y_{2,s}',\ldots,Y_{n,s}')\right)\right| \right]
\\
& \le &
\EXP \left[ \sup_{s\in S^{d-1}} \left|f(Y_{2,s},\ldots,Y_{n,s})
      - f(Y_{2,s}',\ldots,Y_{n,s}') \right| \right] \\
& = &
\EXP \left[ \sup_{s\in S^{d-1}}
\left|f(\epsilon_2(Y_{2,s},Y_{2,s}'),\ldots,\epsilon_n(Y_{n,s},Y_{n,s}'))
  - f(\epsilon_2(Y_{2,s}',Y_{2,s}),\ldots,\epsilon_n(Y_{n,s}',Y_{n,s}))
\right| \right]~.
\end{eqnarray*}
Introduce now the expectation operator $\EXP_{\epsilon}$ that computes
expectation with respect to the random swaps only. Then we can further bound
the expectation above by
\[
2 \EXP \EXP_{\epsilon} 
\left[ \sup_{s\in S^{d-1}}
\left|f(\epsilon_2(Y_{2,s},Y_{2,s}'),\ldots,\epsilon_n(Y_{n,s},Y_{n,s}'))
  - \EXP_{\epsilon} f(\epsilon_2(Y_{2,s},Y_{2,s}'),\ldots,\epsilon_n(Y_{n,s},Y_{n,s}'))
\right| \right]~.
\]
Next we bound the inner expectation. Note that for fixed
$\bX_n,\bX_n'$, by Lemma \ref{lem:schlaffli}, there are at most
$n^{2d}$ different dichotomies of the $2\binom{n}{2}$ points in
$\bX_n\cup\bX_n'$ by hyperplanes including the origin and therefore 
there are not more than $n^{2d}$ random variables of the form
$f(\epsilon_2(Y_{2,s},Y_{2,s}'),\ldots,\epsilon_n(Y_{n,s},Y_{n,s}'))$ as
$s$ varies over $S^{d-1}$. On the other hand, for any fixed $s$, 
the value of
$f(\epsilon_2(Y_{2,s},Y_{2,s}'),\ldots,\epsilon_n(Y_{n,s},Y_{n,s}'))$ can
change by at most $1$ if one flips the value of one of the
$\epsilon_i(Y_{i,s},Y_{i,s}')$ ($i=2,\ldots,n$), since such a flip
amounts to changing the edges incident to vertex $i$ and therefore
can change the value of the chromatic number by at most one.
Thus, by the bounded differences inequality (see, e.g.,
\cite[Section 6.1]{BoLuMa13}), for all $s\in S^{d-1}$ and $\lambda>0$,
\begin{eqnarray*}
  \EXP_{\epsilon} \left[\exp\left(\lambda ( f(\epsilon_2(Y_{2,s},Y_{2,s}'),\ldots,\epsilon_n(Y_{n,s},Y_{n,s}'))
  - \EXP_{\epsilon}
  f(\epsilon_2(Y_{2,s},Y_{2,s}'),\ldots,\epsilon_n(Y_{n,s},Y_{n,s}')))
\right) \right] \\
\le \exp\left(\frac{(n-1)\lambda^2}{2}\right)~.
\end{eqnarray*}
Therefore, by a standard maximal inequality for sub-Gaussian random
variables (\cite[Section 2.5]{BoLuMa13}), 
\begin{eqnarray*}
\EXP_{\epsilon} 
\left[ \sup_{s\in S^{d-1}}
\left|f(\epsilon_2(Y_{2,s},Y_{2,s}'),\ldots,\epsilon_n(Y_{n,s},Y_{n,s}'))
  - \EXP_{\epsilon} f(\epsilon_2(Y_{2,s},Y_{2,s}'),\ldots,\epsilon_n(Y_{n,s},Y_{n,s}'))
\right| \right]  \\
\le \sqrt{4(n-1)d\log n}~.
\end{eqnarray*}
Since the upper bound is $o(n/\log n)$ for $d=o(n/\log^3 n)$, the
result follows.
\end{proof}

Parts (ii) and (iv) of Theorem \ref {thm:chromatic} follow from the next, straingforward proposition by setting $k=\lfloor (1-\epsilon)n/(2\log_2
n)\rfloor$
and $k'=\lceil (1+\epsilon)n/(2\log_2 n)\rceil$.

\begin{proposition}
Let $k,k'\le n$ be positive integers. 
If $d\ge k\binom{\lceil n/k \rceil}{2}$, then, with probability one, there exists $s\in S^{d-1}$
such that
$\chi(\Gamma(\bX_n,s)) \le k$.
On the other hand, if $d\ge \binom{k'}{2}$,  then, with probability one, there exists $s\in S^{d-1}$
such that
$\chi(\Gamma(\bX_n,s)) \ge k'$.
\end{proposition}

\begin{proof}
Partition the vertex set $[n]$ into $k$ disjoint sets of
size at most $\lceil n/k\rceil$ each. If for some $s\in S^{d-1}$ each of these sets is an
independent set (i.e., contain no edge joining two vertices within the
set) in $\Gamma(\bX_n,s)$, then the graph $\Gamma(\bX_n,s)$
is clearly properly colorable with $k$ colors. 
Let $A$ be the set of pairs of vertices $(i,j)$ such that $i$ and $j$
belong to the same set of the partition. By Lemma \ref{lem:schlaffli},
if
$d\ge k\binom{\lceil n/k \rceil}{2}\ge |A|$, the set of points
$\{X_{i,j}: (i,j)\in A\}$ is shattered by half spaces. In particular,
there exists an $s\in S^{d-1}$ such that $\inner{X_{i,j}}{s}<0$ for
all $(i,j)\in A$ and therefore $\Gamma(\bX_n,s)$ has no edge between
any two vertices in the same set. The first statement follows.

To prove the second statement, simply notice that
is a graph has a clique of size $k$ then its chromatic number at least
$k$. But if $d\ge \binom{k}{2}$, then, by Lemma \ref{lem:schlaffli}, for some
$s\in S^{d-1}$, the vertex set $\{1,\ldots,k\}$ forms a clique.
\end{proof}

It remains to prove Part (iii) of Theorem \ref{thm:chromatic}. To this
end, we combine the covering argument used in parts (i) and (iii) of
Theorem \ref{thm:clique} with a result of Alon and Sudakov
\cite{AlSu10}
(see Proposition \ref{prop:AlonSudakov} below)
that bounds the ``resilience'' of the chromatic number of a random graph.

Let $\C_{\eta}$ be a minimal $\eta$-cover of
$S^{d-1}$ where we take $\eta =c\epsilon^2/(\sqrt{d}\log^2n)$ for a
sufficiently small positive constant $c$. Then
\begin{eqnarray*}
\lefteqn{
\PROB\left\{ \exists s\in S^{d-1}: \chi(\Gamma(\bX_n,s)) > (1+\epsilon)\frac{n}{2\log_2n}\right\} } \\
& \le &
|\C_{\eta}| \PROB\left\{ \exists 
s\in
  S^{d-1}: \|s-s_0\| \le \eta: \chi(\Gamma(\bX_n,s)) > (1+\epsilon)\frac{n}{2\log_2n}  \right\}
\end{eqnarray*}
where $s_0=(1,0,\ldots,0)$.
By the argument used in the proof of parts (i) and (iii) of Theorem
\ref{thm:clique}, 
\[
\bigcup_{s\in  S^{d-1}: \|s-s_0\| \le \eta} \Gamma(\bX_n,s)
\subset \Gamma(\bX_n,s_0) \cup E
\]
where $E$ is a set of $\text{Bin}(\binom{n}{2},\alpha_n)$ edges where, 
$\alpha_n= \frac{\eta\sqrt{d}}{\sqrt{2\pi}}$.
By our choice of $\eta$, we have 
$\alpha_n \le c_2\epsilon^2n^2/(\log_2n)^2$ where $c_2$ is the
constant appearing in Proposition \ref{prop:AlonSudakov}. Thus, by the
Chernoff bound,
\[
  \PROB\left\{  |E| >  \frac{c_2\epsilon^2n^2}{(\log_2n)^2}\right\}
\le \exp\left(- \frac{c_2(\log 2-1/2)\epsilon^2n^2}{(\log_2n)^2} \right)~.
\]
Hence, by Proposition \ref{prop:AlonSudakov},
\begin{eqnarray*}
\PROB\left\{ \exists 
s\in
  S^{d-1}: \|s-s_0\| \le \eta: \chi(\Gamma(\bX_n,s)) >
  (1+\epsilon)\frac{n}{2\log_2n}  \right\}
\\
\le \exp\left(- \frac{c_2(\log 2-1/2)\epsilon^2n^2}{(\log_2n)^2}\right)
+ \exp\left(- \frac{c_1n^2}{(\log_2n)^4}\right)~.
\end{eqnarray*}
Combining this bound with Lemma \ref{lem:cover} implies the
statement. \qed

\section{Connectivity}
\label{sec:connectivity}

In this section we study connectivity of the random graphs in
$\G_{d,p}(\bX_n)$. It is well known since the pioneering work of
Erd\H{o}s and R\'enyi \cite{ErRe60} that the threshold for
connectivity for a $G(n,p)$ random graph is when $p=c\log n/n$. 
For $c<1$, the graph is disconnected and for $c>1$ it is connected,
with high probability. In this section we investigate both regimes.
In particular, for $c>1$ we establish lower and upper bounds for the
smallest dimension $d$ such that some graph in $\G_{d,c\log  n/n}(\bX_n)$ 
is disconnected. We prove that this value of $d$ is of the order of
$(c-1)\log n/\log\log n$. 
For the regime $c<1$ we also establish lower and upper bounds for the
smallest dimension $d$ such that some graph in $\G_{d,c\log n/n}(\bX_n)$ 
is connected. As in the case of the clique number and chromatic
number, here as well we observe a large degree of asymmetry. In order to
witness some connected graphs in $\G_{d,c\log n/n}(\bX_n)$, the
dimension $d$ has to be at least of the order of $n^{1-c}$. While we
suspect that this bound is essentially tight, we do not have a
matching upper bound. However, we are able to show that when $d$ is of
the order of $n\log n$, the family $\G_{d,c\log  n/n}(\bX_n)$ not only
contains connected graphs, but also, with high probability, for every
spanning tree of the vertices $[n]$, there exists an $s\in S^{d-1}$
such that $\Gamma(\bX_n,s,t)$ contains the spanning tree. (Recall that $t$ is
such that $p=1-\Phi(t)$.)

\begin{theorem}
\label{thm:connectivity}
{\sc (connectivity.)} Assume $p=c\log n/n$ and let $t=\Phi^{-1}(1-p)$.
Then with high probability the following hold:
\begin{itemize}
\item[(i)] {\sc (subcritical; necessary.)}
If $c<1$ then for any   $\epsilon \in (0,c)$, if
$d=O(n^{1-c-\epsilon})$, then for all $s\in S^{d-1}$,
$\Gamma(\bX_n,s,t)$ is disconnected.
\item[(ii)] {\sc (subcritical; sufficient.)}
There exists an absolute constant $C$ such that if $d\ge Cn\sqrt{\log  n}$,
then there exists an $s\in S^{d-1}$ such that
$\Gamma(\bX_n,s,t)$ is connected.
\item[(iii)] {\sc (supercritical; necessary.)}
If $c>1$ then for any $\epsilon>0$, if $d\le (1-\epsilon)(c-1)\log
n/\log\log n$, then for all $s\in S^{d-1}$,
$\Gamma(\bX_n,s,t)$ is connected.
\item[(iv)] {\sc (supercritical; sufficient.)}
If $c>1$ then for any $\epsilon>0$, if $d\ge (2+\epsilon)(c-1)\log
n/\log\log n$, then for some $s\in S^{d-1}$,
$\Gamma(\bX_n,s,t)$ is disconnected.
\end{itemize}
\end{theorem}

\subsubsection*{Proof of Theorem \ref{thm:connectivity}, part (i).}
To prove part (i), we show that when $d=O(n^{1-c-\epsilon})$, with
high probability,
all graphs $\Gamma(\bX_n,s,t)$ contain at least one isolated point.
The proof of this 
 is based on a covering argument similar those
used in parts of Theorems \ref{thm:clique} and \ref{thm:chromatic},
combined with a sharp estimate for the probability that 
$G(n,c\log n/n)$ has no isolated vertex. This estimate, given in
Lemma \ref{lem:oconnell} below, is proved by an elegant argument of 
O'Connell \cite{oconnell}. 

Let $\eta\in (0,1]$ to be  specified below and let $\C_{\eta}$ be a minimal $\eta$-cover of
$S^{d-1}$. If $N(s)$ denotes the number of isolated vertices (i.e.,
vertices of degree $0$) in $\Gamma(\bX_n,s,t)$,
then
\begin{eqnarray*}
\lefteqn{
\PROB\left\{ \exists s\in S^{d-1}:  \Gamma(\bX_n,s,t) \ \text{is connected}\right\} } \\
& \le & \PROB\left\{ \exists s\in S^{d-1}: N(s)=0 \right\}  \\
& \le &
|\C_{\eta}|\PROB\left\{ \exists 
s\in
  S^{d-1}: \|s-s_0\| \le \eta: N(s)=0  \right\}
\end{eqnarray*}
where $s_0=(1,0,\ldots,0)$. It follows by the first half of Lemma \ref{lem:capsandprobs} 
that there exists a constant $\kappa>0$ such that if $\eta=\kappa
\epsilon/(t\sqrt{d})$, then
\[
\PROB\left\{ \exists 
s\in
  S^{d-1}: \|s-s_0\| \le \eta: N_k(s)=0  \right\}
\le \PROB\left\{ N=0  \right\}
\]
where $N$ is the number of isolated vertices in a
$G(n,(c+\epsilon/2)\log n/n)$ random graph. By Lemma \ref{lem:oconnell},
for $n$ sufficiently large, this is at most $\exp(-n^{-(1-c-\epsilon/2)}/3)$.
Bounding $|\C_{\eta}|$ by Lemma \ref{lem:cover} and substituting the
chosen value of $\eta$ proves part (i).

\medskip
\noindent
\subsubsection*{Proof of Theorem \ref{thm:connectivity}, part (ii).}
Part (ii) of Theorem \ref{thm:connectivity} follows from a
significantly more general statement. Based on a geometrical argument,
we show that for any positive integer $k$, if $d$ is at least a
sufficiently large constant multiple of $k\Phi^{-1}(1-p)$, then with
high probability, $k$
independent standard normal vectors in $\R^d$ are shattered by half
spaces of the form $\{x: \inner{x}{s}\ge t\}$. 
In particular, by taking $k=n-1$ and considering the normal vectors
$X_{i,j}$ corresponding to the edges of any fixed spanning tree, one finds an $s\in S^{d-1}$
such that $\Gamma(\bX_n,s,t)$ contains all edges of the spanning tree,
making the graph connected. Note that if $d\ge Cn\sqrt{\alpha \log n}$ then
the same statement holds whenever $p=n^{-\alpha}$ regardless of how
large $\alpha$ is. Thus, for $d\gg n\sqrt{\log n}$, some
$\Gamma(\bX_n,s,t)$ are connected, even though for a typical $s$, the
graph is empty with high probability.

Fix a set $E$ of edges of the complete graph $K_n$. We say that
$\G_{d,p}(\bX_n)$ shatters $E$ if 
$\{e(G)\, : G \in \G_{d,p}(\bX_n)\}$ shatters $E$
(where $e(G)$ denotes the set of edges of a graph $G$).
In other words, $\G_{d,p}(\bX_n)$ shatters $E$ if 
for all $F \subset E$ there is $G \in \G_{d,p}(\bX_n)$ such that $e(G) \cap E=F$. 
\begin{proposition}\label{prop:rv}
Fix $n \in \mathbb{N}$, $k \in \{1,2,\ldots,{n \choose 2}\}$, and a
set $E=\{e_1,\ldots,e_k\}$ of edges of the complete graph $K_n$. 
There exist universal constants $b,c > 0$ such that for 
$d \ge (4/c)\cdot k \cdot \Phi^{-1}(1-p)$ we have 
\[
\PROB\left(\G_{d,p}(\bX_n)~\mbox{shatters}~E\right) \ge 1-e^{-bd}~. 
\]
\end{proposition}
\begin{proof}
Given points $x_1,\ldots,x_k$ in $\mathbb{R}^d$, the {\em affine span} of $x_1,\ldots,x_k$ is the set $\{\sum_{i=1}^k c_i X_i\, : \sum_{i=1}^k c_i =1\}$. 
 Fix $E=\{e_1,\ldots,e_k\} \in \mathcal{S}_k$ and let $P_E$ be the affine span of $X_{e_1},\ldots,X_{e_k}$. 
Also, let  $t = \Phi^{-1}(1-p)$.

First suppose that $\min\{\|y\|:y \in P_E\} > t$. Then we may shatter $E$ as follows. 
First, almost surely, $P_E$ is a $(k-1)$-dimensional affine subspace in $\mathbb{R}^d$. 
Assuming this occurs, then $E$ is shattered by halfspaces in $P_E$: in other words, 
for any $F \subset E$ there is a $(k-2)$-dimensional subspace $H$ contained within $P_E$ such that $F$ and $E\setminus F$ lie on opposite sides of $H$ in $P_E$ (i.e., in different 
connected components of $P_E\setminus H$). 

Fix $F \subset E$ and $H \subset P_E$ as in the preceding paragraph. Then 
let $K$ be a $(d-1)$-dimensional hyperplane tangent to $tS^{d-1} = \{x \in \mathbb{R}^d: \|x\|=t\}$, intersecting $P_E$ only at $H$, and separating the origin 
from $F$. In other words, $K$ is such that $K \cap P_E = H$ and $|K \cap tS^{d-1}| = 1$, and also such that $0$ and $F$ lie on opposite sides of $K$
of $\R^d\setminus K$. Since $P_E$ has dimension $k-1 < d-2$, such a hyperplane $K$ exists. Since $F$ and $E\setminus F$ lie on opposite sides of $H$, 
we also obtain that $0$ and $E\setminus F$ lie on the {\em same} side of $K$. 

Let $s \in S^{d-1}$ be such that $ts \in K$. Then for $e \in F$ we have $\inner{X_{e}}{s} > t$, and for $e \in E \setminus F$ we have $\inner{X_e}{s} < t$. 
It follows that $E \cap \Gamma(\bX,s,t) = F$. Since $F \subset E$ was arbitrary, this implies that 
\[
\PROB(\G_{d,p}(\bX_n) \mbox{ shatters } E) \ge \PROB(\min\{\|y\|:y \in P_E\} > \Phi^{-1}(1-p))\, ,
\]
In light of the assumption that $d \ge (4/c)\cdot k \cdot
\Phi^{-1}(1-p)$, the proposition is then immediate from
Lemma~\ref{lem:rv} below.
\end{proof}

The key element of the proof of Proposition \ref{prop:rv} is that
the affine span of $k\le 4d$ independent standard normal vectors in
$\R^d$ is at least at distance of the order of $d/k$ from the origin. 
This is made precise in the following lemma whose proof crucially uses
a sharp estimate for the smallest singular value of a $d\times k$
Wishart matrix, due to Rudelson and Vershynin \cite{rv}.

\begin{lemma}\label{lem:rv}
There exist universal constants $b,c > 0$ such that the following holds. 
Let $N_1,\ldots,N_k$ be independent standard normal vectors in $\mathbb{R}^d$, let $P$ be the affine span of $N_1,\ldots,N_k$, and let 
$D = \min\{ \|y\|: y \in P\}$. Then whenever $d \ge 4k$, we have $\PROB(D \le cd/4k) < 2e^{-bd}$. 
\end{lemma}
\begin{proof}
We use the notation $\by= (y_1,\ldots,y_k)$. We have 
\begin{align*}
D &=  \min_{\by\, :\sum y_i=1} \left\|\sum_{i=1}^k y_iN_i\right\|^2 \\
& = \min_{\by\, :\sum y_i=1} |\by|^2 \left\|\sum_{i=1}^k \frac{y_i}{\|\by\|} N_i\right\|^2 \\
& \ge \min_{\by\, :\sum y_i=1} \frac{1}{k} \left\|\sum_{i=1}^k \frac{y_i}{\|\by\|} N_i\right\|^2 \\
& \ge \frac{1}{k} \min_{\by\, :|\by|^2=1} \left\|\sum_{i=1}^k y_i N_i\right\|^2\, ,
\end{align*}
where the first inequality holds because if $\sum_{i=1}^k y_i = 1$ then $\|\by\|^2 \ge k^{-1}$ and the second by noting that 
the vector $(y_i/\|\by\|,1 \le i \le k)$ has $2$-norm $1$. 

Let $\bN$ be the $d\times k$ matrix with columns $N_1^t,\ldots,N_k^t$, and write $\bN=(N_{ij})_{ij \in [d]\times[k]}$. 
Then 
\begin{align*}
\min_{\by\, :|\by|^2=1} \left\|\sum_{i=1}^k y_i N_i\right\|^2 
& = \left(\min_{\by\, :|\by|^2=1} \left\|\sum_{i=1}^k y_i N_i\right\|\right)^2 \\
& = \left(\min_{\by\, :|\by|^2=1} \|\bX y\|\right)^2\, .
\end{align*}
The final quantity is just the square of the least singular value of
$\bX$. Theorem 1.1 of Rudelson and Vershynin \cite{rv}
states the existence of absolute constants $b,B>0$ such that for every $\varepsilon > 0$ we have 
\[
\PROB\left(\min_{\by\, :|\by|=1} \|\bX y\| \le \varepsilon (\sqrt{d}-\sqrt{k-1})\right)
\le (B\varepsilon)^{(d-k+1)} + e^{-bd}\, .
\]
If $d \ge 4(k-1)$ then $\sqrt{d}-\sqrt{k-1} \ge \sqrt{d}/2$ and $d-k+1 > d$. 
Combining the preceding probability bound with the lower bound on $D$, if $\varepsilon \le e^{-b}/B$ we then obtain 
\[
\PROB\left(D < \varepsilon^2\frac{d}{4k}\right) < 2e^{-bd}. 
\]
Taking $c = (e^{-b}/B)^2$ completes the proof. 
\end{proof}

One may now easily use Proposition \ref{prop:rv} to deduce 
part (ii) of Theorem \ref{thm:connectivity}:

\begin{proposition}
\label{prop:spanningtree}
There are absolute constants $b,C>0$ such that the following holds. For all $p \le 1/2$, if $d \ge Cn\sqrt{\log(1/p)}$ then with probability 
at least $1-e^{-bd}$ there exists $s \in S^{d-1}$ such that $\Gamma(\bX,s,\Phi^{-1}(1-p))$ is connected. 
\end{proposition}
\begin{proof}
Fix any tree $T$ with vertices $[n]$, and write $E$ for the edge set of $T$. 
By Proposition~\ref{prop:rv}, if $d \ge (4/c)\cdot k \cdot \Phi^{-1}(1-p)$ then with probability at least $1-e^{-bd}$ there is $s$ such that 
$\Gamma(\bX,s,\Phi^{-1}(1-p))$ contains $T$, so in particular is connected. 
Now simply observe that for $p \le 1/2$ we have $\Phi^{-1}(1-p) \le \sqrt{2 \log(1/p)}$.   
\end{proof}

Observe that the exponentially small failure probability stipulated in
Proposition \ref{prop:spanningtree} allows us to conclude that if $d$
is at least a sufficiently large constant multiple of $n(\log n \vee
\sqrt{\log(1/p)})$, 
then, with high probability, for \emph{any} spanning tree of the complete graph
$K_n$ there exists $s\in S^{d-1}$ such that $\Gamma(\bX,s,\Phi^{-1}(1-p))$
contains that spanning tree.

\medskip
\noindent
\subsubsection*{Proof of Theorem \ref{thm:connectivity}, part (iii).}
Let $c>1$, $\epsilon\in (0,1)$, and assume that $d\le (1-\epsilon)(c-1)\log
n/\log\log n$. Let $E$ be the event that $\Gamma(\bX_n,t,s)$ is
disconnected for some $s\in S^{d-1}$. 
Let $\C_{\eta}$ be a minimal $\eta$-cover of
$S^{d-1}$ for $\eta\in (0,1]$ to be  specified below. Then
\[
   E \subseteq \bigcup_{s\in \C_{\eta}} E_s~,
\]
where $E_s$ is the event that the graph $\bigcap_{s':\|s-s'\|\le \eta}\Gamma(\bX_n,t,s')$
is disconnected.
Let $c'=c-(c-1)\epsilon/2$. Note that $1<c'<c$.
It follows from the second half of Lemma \ref{lem:capsandprobs} 
that there exists a constant $\kappa>0$ such that if $\eta=\kappa
(1-c'/c)/(t\sqrt{d})$, then 
\[
   \PROB\left\{E_s\right\} \le \PROB\left\{G(n,c'\log n/n) 
\    \text{is disconnected} \right\}\le n^{1-c'}(1+o(1))~,
\]
where the second inequality follows from standard estimates for the
probability that a random graph is disconnected, see Palmer \cite[Section 4.3]{Pal85}.
Bounding $|\C_{\eta}|$ by Lemma \ref{lem:cover}, and using the fact
that $t=\sqrt{2\log n}(1+o(1))$, 
we obtain that
\begin{eqnarray*}
  \PROB\{E\} & \le & |\C_{\eta}|n^{1-c'}(1+o(1))  \\
 &  \le & \exp\left( \frac{d\log\log n}{2}+ \frac{d\log d}{2}+O(d)+
     (1-c')\log n \right) \to 0~,
\end{eqnarray*}
as desired.

\medskip
\noindent
\subsubsection*{Proof of Theorem \ref{thm:connectivity}, part (iv).}
Recall that $p = c\log n/n$ for $c >1$ fixed, and that $t = \Phi^{-1}(p)$. 
Let $0<\epsilon<1$, and assume that $d\ge (2+\epsilon)(c-1)\log
n/\log\log n$.  Define $\theta \in
(0,\pi/2)$ by $\theta=(\log n)^{-1/(2+\e)}$, soæ that $\log (1/\theta)= \log\log n/(2+\epsilon)$.
Let $\cP$ be a maximal $\theta$-packing of $S^{d-1}$, that is,
$\cP\subset S^{d-1}$ is a set of maximal cardinality 
such that for all distinct $s,s'\in \cP$ 
we have $\inner{s}{s'}\le \cos\theta$.
By Lemma \ref{lem:pack} we have that
\[
|\cP| \ge   \frac{d}{16}\theta^{-(d-1)}~.
\]
It suffices to prove that for some $s\in \cP$,
$\Gamma(\bX_n,s,t)$ contains an isolated vertex. 

For each $s\in \cP$, we write the number of isolated vertices in $\Gamma(\bX_n,s,t)$ 
as
\[
N(s)= \sum_{i=1}^n \prod_{j:j\neq i} Z_{i,j}(s),
\]
where $Z_{i,j}(s)$ equals $1$ if $\{i,j\}$ is not an edge in
$\Gamma(\bX_n,s,t)$  and is $0$ otherwise.  
We use the second moment method to prove that $N\defeq
\sum_{s\in\cP} N(s) >0$ with high probability. This will establish the assertion 
of part (iv) since if $N > 0$ then there is $s \in S^{d-1}$ such that $\Gamma(\bX_n,s,t)$ 
contains an isolated vertex. 

To show that $N > 0$ with high probability, by the second moment method 
it suffices to prove that $\EXP N\to \infty$ and that $\EXP[N^2] =
(\EXP N)^2(1+o(1)$.
First, 
\[
\EXP N = |\cP|\cdot n \cdot 
  \PROB\left\{\text{vertex}\ 1 \ \text{is isolated in} \ G(n,p) \right\} 
= |\cP|\cdot n(1-p)^{n-1}. 
\]
The lower bound on $|\cP|$ and the inequality $1-p \le e^{-p} = n^{-c/n}$ together imply
\[
\EXP N \ge \frac{d}{16} \theta^{-(d-1)} n^{1-c},
\]
which tends to infinity by our choice of $\theta$.
We now turn to the second moment. 
\[
\EXP[N^2] = \sum_{s,s'\in \cP} \sum_{i,j\in [n]} \prod_{k:k \ne i, \ell:\ell \ne j}Z_{i,k}(s)Z_{j,\ell}(s')~.
\]
When $s=s'$, separating the inner sum into diagonal and off-diagonal terms yields the identity 
\[
\sum_{i,j} \prod_{k \ne i,\ell \ne j} Z_{i,k}(s)Z_{j,\ell}(s)
= n(1-p)^{n-1} + n(n-1)(1-p)^{2n-3}
= n(1-p)^{n-1}\cdot [1+(n-1)(1-p)^{n-2}]\, .
\]
Let $q = \sup_{s \ne s', s,s' \in \cP} \PROB\{Z_{i,j}(s)Z_{i,j}(s')=1\}$ 
be the greatest probability that an edge is absent in both
$\Gamma(\bX_n,s,t)$ 
and $\Gamma(\bX_n,s',t)$. Then when $s\ne s'$, the inner sum is bounded by 
\[
n q^{n-1} + n(n-1)\cdot q \cdot (1-p)^{2n-4}.  
\]
Combining these bounds, we obtain that 
\[
\EXP[N^2] \le 
|\cP| n(1-p)^{n-1}\cdot [1+(n-1)(1-p)^{n-2}] + 
|\cP| (|\cP|-1) \cdot [ n q^{n-1} + n(n-1)\cdot q \cdot (1-p)^{2n-4} ]\, .
\]
The first term on the right is at most $\EXP N (1+\EXP N/[(1-p)|\cP|])$. 
The second is at most 
\[
|\cP|^2 n^2 (1-p)^{2(n-1)} \cdot \left( \frac{1}{n} \Big(\frac{q}{(1-p)^2}\Big)^{n-1} + \frac{q}{(1-p)^2} \right) = (\EXP N)^2 \cdot \left( \frac{1}{n} \Big(\frac{q}{(1-p)^2}\Big)^{n-1} + \frac{q}{(1-p)^2} \right)\, .
\]

We will show below that $q \le (1-p)^2\cdot (1+o(p))$. Assuming this, the upper bounds on the two terms on the right together give 
\[
\frac{\EXP[N^2]}{[\EXP N]^2} \le \frac{1}{\EXP N} \left(1+\frac{\EXP N}{(1-p)|\cP|}\right) + 
n^{o(1)-1} + \frac{(1-\epsilon)\log n}{n}\,  \to 1\, ,
\]
as required. 

To prove the bound on $q$, fix $s,s' \in \cP$ such that $q = \PROB\{Z_{i,j}(s)Z_{i,j}(s')=1\}$. 
Using the definition of $Z_{i,j}(s)$ and $Z_{i,j}(s')$, we have 
\[
q = \PROB\left\{ \{i,j\} \not\in \Gamma(\bX_n,s,t), \{i,j\} \not\in \Gamma(\bX_n,s',t),\right\}\, .
\]
We may apply Lemma~\ref{lem:correlations} to this quantity, noting that in our case $\theta = (\ln n)^{1/(2+\e)}$, $t=O(\sqrt{\ln n})$ and $\ln(1/t\,p) = (1+o(1))\,\ln n\gg \theta^{-2}$. This means that the Remark after the statement of the Lemma applies, and this gives precisely that $q\leq (1-p)^2\,(1+o(p))$, as desired.
\qed

\section{Appendix}

Here we gather some of the technical tools used in the paper. In the
first section we summarize results involving covering and packing
results of the unit sphere that are essential in dealing with the
random graph process $\G_{d,1/2}(\bX_n)$. In Section
\ref{sec:caplemmas} we describe analogous results needed for 
studying $\G_{d,p}(\bX_n)$ for small values of $p$. These lemmas 
play an important role in the proof of Theorem \ref{thm:connectivity}.
Finally, in Section \ref{sec:randomgraphlemmas} we collect some
results on $G(n,p)$ random graphs needed in our proofs.

\subsection{Covering and packing}\label{sec:cov_pack}

Let $B(a,b) = \int_0^1 t^{a-1}(1-t)^{b-1}\mathrm{d} t$ be the beta function, and 
let $I_x(a,b)$ be the incomplete beta function, 
\[
I_x(a,b) = \frac{\int_0^x t^{a-1}(1-t)^{b-1}\mathrm{d} t}{B(a,b)}\, .
\]
For $\alpha \in [0,\pi]$ and $s \in \mathbb{S}^{d-1}$, let 
\[
C_\alpha(s) = \{s' \in S^{d-1}: \inner{s}{s'} \ge \cos{\alpha}\}
\]
be the cap in $S^{d-1}$ consisting of points at angle at most $\alpha$ from $s$. 
For $\alpha\le \pi/2$ the area of this cap (see, e.g., \cite{li11cap}) is 
\begin{equation}\label{eq:cap_area}
|C_\alpha(s)| = \frac{|S^{d-1}|}{2} \cdot I_{\sin^2 \theta}\left(\frac{d-1}{2},\frac{1}{2}\right)\, .
\end{equation}

We use the following standard estimate of the covering numbers of the 
Euclidean sphere (see, e.g., \cite[Lemma 13.1.1]{mat02}).
\begin{lemma}
\label{lem:cover}
For any $\eta \in (0,1]$ there exists a subset $\C_{\eta}$ of $S^{d-1}$
of size at most $(4/\eta)^d$ such that for all $s\in S^{d-1}$ there
exists $s'\in \C_{\eta}$ with $\|s-s'\| \le \eta$.
\end{lemma}
We now provide a rough lower bound on the number of points that can be packed in $S^{d-1}$ 
while keeping all pairwise angles large. 
\begin{lemma}
\label{lem:pack}
For any $\theta \in (0,\pi/2)$ there exists a subset 
$\cP_{\theta}$ of $S^{d-1}$
of size at least 
\[
  \frac{d}{16}\theta^{-(d-1)}
\]
 such that for all distinct $s,s'\in \cP_{\theta}$ 
we have $\inner{s}{s'} \le \cos\theta$. 
\end{lemma}
\begin{proof}
First note that it suffices to consider $\theta <
1/2$ because otherwise the first bound dominates. Consider $N$
independent standard normal vectors $X_1,\ldots,X_N$. Then 
$U_i=X_i/\|X_i\|$ ($i=1,\ldots,N$) are independent, uniformly
distributed on $S^{d-1}$.
Let
\[
   Z= \sum_{i=1}^N \IND{\min_{j:j\neq i} |\inner{U_i}{U_j}|\le \cos(\theta)}.
\]
Denoting $\PROB\{|\inner{U_i}{U_j}| > \cos(\theta)\}= \phi$, 
\[
   \EXP Z = N(1-\phi)^N \ge N(1-\phi N) \ge N/2
\]
whenever $\phi N \le 1/2$. Since $Z\le N$, this implies that
\[
   \PROB\left\{ Z \ge \frac{N}{4} \right\} \ge \frac{\EXP Z- N/4}{N-
     N/4} \ge \frac{1}{3}
\]
and therefore there exists a packing set $A$ of cardinality $|A|\ge
N/4$ as long as $\phi N \le 1/2$. To study $\phi$, note that
\begin{eqnarray*}
\phi =  \PROB\left\{ \frac{\sum_{j=1}^d Y_jY_j'}{\|Y\|\cdot \|Y'\|}
  > \cos(\theta)\right\}
\end{eqnarray*}
where $Y=(Y_1,\ldots,Y_d),Y'=(Y_1',\ldots,Y_d')$ are independent
  standard normal vectors. By rotational invariance, we may replace 
$Y'$ by $(\|Y'\|,0,\ldots,0)$, and therefore
\begin{eqnarray*}
\phi & = & \PROB\left\{ \frac{Y_1^2}{\|Y\|}
  > \cos^2(\theta)\right\} \\
& = &
\PROB\left\{ B \le \cos^2(\theta)\right\}  \\
& & \text{(where $B$ is a Beta$(1/2,(d-1)/2)$ random variable)} \\
& \ge &
\frac{2 \theta^{d-1}}{d-1}~.
\end{eqnarray*}
The result follows.
\end{proof}

The next lemma is used repeatedly in the proof of Theorem
\ref{thm:clique} and \ref{thm:chromatic}.

\begin{lemma}
\label{lem:caps_half}
Fix $s'\in S^{d-1}$ and $\eta \in (0,1]$ and assume that $d\ge 12$. The probability that there
exists $s\in S^{d-1}$ with $\|s-s'\| \le \eta$ such that 
vertex $1$ and vertex $2$ are connected in $\Gamma(\bX_n,s)$ but 
not in $\Gamma(\bX_n,s')$ is at most
\[
 \eta\sqrt{\frac{d}{2\pi}}~.
\]
\end{lemma}

\begin{proof}
Without loss of generality, assume that $s'=(1,0,\ldots,0)$.
Observe that the event 
that there
exists $s'\in S^{d-1}$ with $\|s-s'\| \le \eta$ such that 
vertex $1$ and vertex $2$ are connected in $\Gamma(\bX_n,s)$ but 
not in $\Gamma(\bX_n,s')$ is equivalent to 
$X_{1,2}/\|X_{1,2}\|$ having its first component  between $-\eta
\sqrt{1-\eta^2/2}$ and $0$
(see Figure \ref{fig:height}).
Letting $Z=(Z_1,\ldots,Z_d)$ be a standard normal vector in $\R^d$, the probability of this is 
\begin{eqnarray*}
\PROB\left\{ \frac{Z_1}{\|Z\|} \in \left(-\eta
  \sqrt{1-\eta^2/2},0\right)\right\}
&\le & 
\PROB\left\{ \frac{Z_1}{\|Z\|} \in \left(-\eta,0 \right)\right\}  \\
& = &
\frac{1}{2} \PROB\left\{ B \le \eta^2\right\}  \\
& & \text{(where $B$ is a Beta$(1/2,(d-1)/2)$ random variable)} \\
& = &
\frac{1}{2} I_{\eta^2}(1/2,(d-1)/2) \\
& \le &
\frac{1}{2B(1/2,(d-1)/2)} \int_0^{\eta^2} x^{-1/2}dx \\
& = &
\frac{\eta}{2B(1/2,(d-1)/2)} \\
& \le &
\eta\sqrt{\frac{d-1}{2\pi}}~. \qedhere
\end{eqnarray*}
\end{proof}

\begin{figure}
\centering
\includegraphics[scale=1]{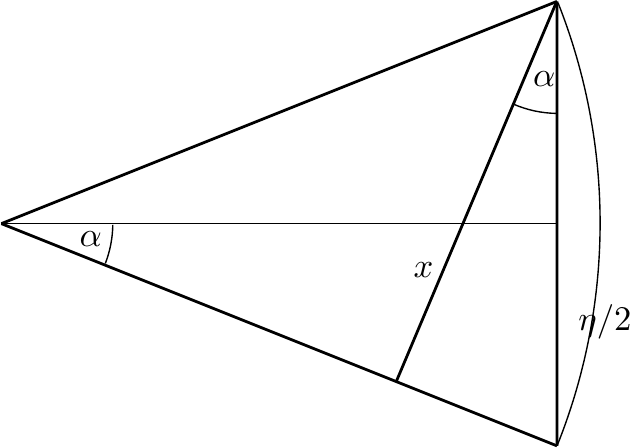}
\caption{\label{fig:height}
Since $\cos(\alpha) = \eta/2 = \sqrt{\eta^2-x^2}/\eta$, the height
of the spherical cap that only includes points at distance at least $\eta$
from the equator is $1-x=1-\eta\sqrt{1-\eta^2/2}$.
 }
\end{figure} 

\subsection{Auxiliary results for $\G_{d,p}(\bX_n)$}
\label{sec:caplemmas}

In this section we develop some of the main tools for dealing with the random graph process  $\G_{d,p}(\bX_n)$. We assume throughout the section that 
\begin{equation}\label{eq:polyp}p:=1-\Phi(t)\leq \frac{1}{2}.\end{equation}
Recall from the start of Section~\ref{sec:cov_pack} that $C_\alpha(s)$ denotes 
the spherical cap consisting of all unit vectors with an angle of $\leq \alpha$ with $s$. 
We will use the following expressions for $C_{\alpha}(s)$: 
\begin{eqnarray}
\label{eq:defcapdistance}
C_\alpha(s)&=& \{s'\in S^{d-1}\,:\, \|s-s'\|^2 \leq 2\,(1-\cos\alpha)\}\nonumber\\
\label{eq:defcapgeometric} &=&\{s\cos\theta + w\sin\theta\,:\,w\in S^{d-1}\cap \{v\}^{\perp},\, 0\leq \theta\leq \alpha\}.\end{eqnarray}

We are interested in studying the graphs $\Gamma(\bX_n,s',t)$, for all
$s'\in C_\alpha(s)$ simultaneously.

\begin{lemma}\label{lem:capsandprobs}There exists a constant $c>0$ such that, for all $\e\in (0,1/2)$, if $t\geq 0$ and $p$ are as in (\ref{eq:polyp}),
$$0\leq \alpha\leq \frac{\pi}{2},\, \tan\alpha\leq \frac{\e}{(t\vee 1)\,\sqrt{d-1}},$$
then, for some universal $c>0$, if we define $\e':=\e + c\,(\e^2 + \e/(t^2\vee 1))$,
\begin{enumerate}
\item the union $\Gamma_+:=\bigcup_{s'\in C_\alpha(s)}\Gamma(\bX_n,s',t)$ is stochastically dominated by $G(n,(1+\e')\,p)$;
\item the intersection $\Gamma_-:=\bigcap_{s'\in C_\alpha(s)}\Gamma(\bX_n,s',t)$ stochastically dominates by $G(n,(1-\e')\,p)$.\end{enumerate}\end{lemma}

\begin{proof}The {\em first step} in this argument is to note that the edges of both $\Gamma_{+}$ and $\Gamma_-$ are independent. To see this, just notice that, for any $\{i,j\}\in\binom{[n]}{2}$, the event that $\{i,j\}$ is an edge in $\Gamma_{\pm}$ depends on $\bX_n$ only through $X_{i,j}$. More specifically,
\begin{eqnarray}\label{eq:edge+}\{i,j\}\in \Gamma_+&\Leftrightarrow &\exists s'\in C_\alpha(s)\,:\, \inner{X_{i,j}}{s'}\geq t;\nonumber\\
 \label{eq:edge-}\{i,j\}\in \Gamma_-&\Leftrightarrow &\forall s'\in C_\alpha(s)\,:\, \inner{X_{i,j}}{s}\geq t.\nonumber\end{eqnarray}

The main consequence of independence is that we will be done once we show that 
\begin{equation}\label{eq:finalgoalcaps}(1-\e')\,p\leq \Pr\{\{i,j\}\in \Gamma_-\}\leq \Pr\{\{i,j\}\in \Gamma_+\}\leq (1+\e')\,p.\end{equation}

As a {\em second step} in our proof, we analyze the inner product of $X_{i,j}$ with $s'=s\cos \theta + w\sin\theta \in C_\alpha(s)$ (with the same notation as in (\ref{eq:defcapgeometric})). Note that
$$\inner{s'}{X_{i,j}} = N\cos\theta + \inner{w}{X_{i,j}^\perp}\sin\theta = \cos\theta\,\left(N + \inner{w}{X_{i,j}^\perp}\tan\theta\right),$$
where $N:=\inner{X_{i,j}}{s}$ and $X_{i,j}^{\perp}$ is the component
of $X_{i,j}$ that is orthogonal to $s$. Crucially, the fact that
$X_{i,j}$ is a standard Gaussian random vector implies that $N$ is a
standard normal random variable and $X_{i,j}^{\perp}$ is an independent standard normal random vector in $s^\perp$. Moreover, 
$$\forall w\in S^{d-1}|\inner{w}{X_{i,j}^\perp}|\leq \chi:=\|X_{i,j}^{\perp}\|.$$
Since ``$\theta \mapsto \tan\theta$" is increasing in $[0,\alpha]$, we conclude
\begin{equation}\label{eq:boundcapall}\forall s'\in C_\alpha(s)\,:\, \inner{s'}{X_{i,j}}=\cos\theta\left( N + \Delta(s')\right), \mbox{ where }|\Delta(s')|\leq \,(\tan\alpha)\,\chi.\end{equation}
Our {\em third step} is to relate the above to the events $\{\{i,j\}\in\Gamma_{\pm}\}$. On the one hand,
\begin{eqnarray*}\{i,j\}\in\Gamma_{+}&\Leftrightarrow &\max_{s'\in C_\alpha(s)}\inner{s'}{X_{i,j}}\geq t\\ 
&\Rightarrow& N + \max_{s'\in\C_\alpha(s)}\Delta(s')\geq t 
\quad \text{(use (\ref{eq:boundcapall}) and $0\leq \cos\theta\leq 1$)}\\ 
&\Rightarrow & N \geq t - (\tan\alpha)\,\chi,\end{eqnarray*}
and we conclude (using the independence of $N$ and $\chi$) that
\begin{equation}\label{eq:upperwithchi}\Pr\{\{i,j\}\in\Gamma_+\}\leq 1 - \E[\Phi(t - (\tan\alpha)\,\chi)].\end{equation}
Similarly,
\begin{eqnarray*}\{i,j\}\in\Gamma_{-}&\Leftrightarrow &\min_{s'\in C_\alpha(s)}\inner{s'}{X_{i,j}}\geq t\\ 
&\Leftrightarrow& N + \min_{s'\in\C_\alpha(s)}\Delta(s')\geq
\frac{t}{\cos\alpha}  
\quad \text{(by (\ref{eq:boundcapall}) and $\cos\theta\geq \cos\alpha>0$)}\\ 
&\Leftarrow & N \geq \frac{t}{\cos\alpha} + (\tan\alpha)\,\chi,\end{eqnarray*}
and we conclude
\begin{equation}\label{eq:lowerwithchi}\Pr\{\{i,j\}\in\Gamma_-\}\geq \E\left[1 - \Phi\left(\frac{t}{\cos\alpha} + (\tan\alpha)\,\chi\right)\right].\end{equation}
The remainder of the proof splits into two cases, depending on whether or not
\begin{equation}\label{eq:conditionforsplit}e^{\frac{5t^2}{8}}\,(1-\Phi(t))\,\geq 1\end{equation}
Note that this condition holds if and only if $t\geq C$ for some $C>0$, as $1-\Phi(t) = e^{-(1+o(1))t^2/2}$ when $t\to +\infty$ and $e^{\frac{5t^2}{8}}\,(1-\Phi(t))=1/2<1$ when $t=0$.\\

{\em Last step when (\ref{eq:conditionforsplit}) is violated.} In this case $t$ is bounded above, so $p>c_0$ for some positive constant $c_0>0$. We combine (\ref{eq:upperwithchi}) and (\ref{eq:lowerwithchi}) with the fact that $\Phi(t)$ is $(2\pi)^{-1/2}$-Lipschitz. The upshot is that 
$$|1-\Phi(t) - \Pr\{\{i,j\}\in\Gamma_\pm\}|\leq
\frac{1}{\sqrt{\pi}}\,\left|1 - \frac{1}{\cos\alpha}\right|\,t +
\E[\chi] \tan\alpha.$$
Now $\chi$ is the norm of a $d-1$ dimensional standard normal random vector, so $\E[\chi]\leq \sqrt{\E[\chi^2]}=\sqrt{d-1}$. The choice of $\alpha$ implies:
$$\left|1 - \frac{1}{\cos\alpha}\right| = O(\sin \alpha)=O\left(\frac{\e^2}{d-1}\right)\mbox{, and }\tan\alpha \leq \frac{\e}{\sqrt{d-1}}.$$
So 
$$|1-\Phi(t) - \Pr\{\{i,j\}\in\Gamma_\pm\}|\leq \frac{1}{\sqrt{2\pi}}\,(c\,\e^2+\e)\leq \left[\e + c\,\left(\e^2 + \frac{\e}{t^2}\right)\right]\,p$$
for some universal $c>0$. \\

{\em Last step when (\ref{eq:conditionforsplit}) is satisfied.}~We start with (\ref{eq:lowerwithchi}) and note that we can apply Lemma~\ref{lem:tailsgaussian} with $r:=t$ and
$$h:= \left(\frac{1}{\cos\alpha}-1\right)\,t + (\tan\alpha)\,\chi\leq O((\tan \alpha)^2)\,t + (\tan\alpha)\,\chi.$$
After simple calculations, this gives
\begin{equation}\nonumber\frac{\Pr\{\{i,j\}\in\Gamma_-\}}{1-\Phi(t)}\geq \E\left[\exp\left(-X\right)\right],\end{equation}
where 
$$X:=O((\tan\alpha)^2)\,(t^2+1) - (t+t^{-1})\,(\tan\alpha)\chi -(\tan\alpha)^2\,\xi^ 2 - O((\tan \alpha)^2)t^2.$$
By Jensen's inequality, $\E[e^{-X}]\geq e^{-\E[X]}$. Since $\E[\chi]^2\leq \E[\chi^ 2]=d-1$ and $\tan\alpha = \e/t\,\sqrt{d-1}$ in this case, 
$$\E[X]\leq O\left(\frac{\e^2}{d-1}\right) + (1 + O(\e + t^{-2}))\,\e.$$
In other words, if we choose $c>0$ in the statement of the theorem to be large enough, we can ensure that 
\begin{equation}\nonumber\frac{\Pr\{\{i,j\}\in\Gamma_-\}}{1-\Phi(t)}\geq (1-\e').\end{equation}
We now turn to (\ref{eq:upperwithchi}). Applying Lemma
\ref{lem:tailsgaussian} below with $r:=t - \chi \tan\alpha$ when $r\geq t/2$, we get
\begin{equation}\label{eq:ineqrestrictedchi}1 - \Phi(t-(\tan\alpha)\,\chi)\leq e^{\left(t + \frac{2}{t}\right)\,(\tan\alpha)\,\chi +\frac{(\tan\,\alpha)^2\,\chi^2}{2}}\,(1-\Phi(t)).\end{equation}
In fact, the same inequality holds when $r<t/2$, i.e.,
$(\tan\alpha)\,\chi>t/2$, for in that case the right-hand side is
$\geq e^{\frac{5t^2}{8}}\,(1-\Phi(t))\geq 1$ (recall that we are under
the assumption (\ref{eq:conditionforsplit})). So (\ref{eq:ineqrestrictedchi}) always holds, and integration over $\chi$ gives
\begin{equation}\label{eq:ratiolast}\frac{\Pr\{\{i,j\}\in\Gamma_+\}}{1-\Phi(t)}\leq \E[e^{\left(t + \frac{2}{t}\right)\,(\tan\alpha)\,\chi +\frac{(\tan\,\alpha)^2\,\chi^2}{2}}].\end{equation}
It remains to estimate the moment generating function on the
right-hand side. The first step is to note that, since $\E[\xi]$ is
the norm of a $d-1$ dimensional standard normal vector, $\E[\chi]\leq
\E[\chi^2]^{1/2}=\sqrt{d-1}$. So by Cauchy Schwartz,
\begin{eqnarray}\nonumber e^{-\left(t + \frac{2}{t}\right)\,(\tan\alpha)\sqrt{d-1}}\,\E[e^{\left(t + \frac{2}{t}\right)\,(\tan\alpha)\,\chi +\frac{(\tan\,\alpha)^2\,\chi^2}{2}}] &\leq &  \E[e^{\left(t + \frac{2}{t}\right)\,(\tan\alpha)\,(\chi-\E[\chi]) +\frac{(\tan\,\alpha)^2\,\chi^2}{2}}]\\ \label{eq:twotermsmgf}&\leq &\sqrt{\E[e^{\left(2t + \frac{4}{t}\right)\,(\tan\alpha)\,(\chi-\E[\chi])}] \E[e^{(\tan\,\alpha)^2\,\chi^2}]}.\end{eqnarray}
Next we estimate each of the two expectations on the right-hand side of the last line. In the first case we have the moment generating function of $\chi-\E[\chi]$, where $\chi$ is a $1$-Lipschitz function of a standard Gaussian vector. A standard Gaussian concentration argument and our definition of $\alpha$ give
$$\E[e^{\left(2t + \frac{4}{t}\right)\,(\tan\alpha)\,(\chi-\E[\chi])}]\leq e^{\frac{\left(2t + \frac{4}{t}\right)^2\,(\tan\alpha)^2}{2}}\leq 1 + c_0\e^2$$
for some universal constant $c_0>0$. 
The second tem in (\ref{eq:twotermsmgf}) is the moment generating function of $\chi^2$, a chi-squared random variable with $d-1$ degrees of freedom. Since $(\tan\alpha)^2\leq \e^2/(d-1)\leq 1/2$ under our assumptions, one can compute explicitly
$$ \E[e^{(\tan\,\alpha)^2\,\chi^2}] = \left(\frac{1}{1 - 2(\tan\,\alpha)^2}\right)^{d/2} \leq 1 + c_0\,\e^2$$
for a (potentially larger, but still universal $c_0>0$). Plugging the
two estimates back into (\ref{eq:twotermsmgf}), we obtain
$$ \E[e^{\left(t + \frac{2}{t}\right)\,(\tan\alpha)\,\chi +\frac{(\tan\,\alpha)^2\,\chi^2}{2}}]\leq e^{\left(t + \frac{2}{t}\right)\,(\tan\alpha)\sqrt{d-1}}\,(1+c_0\,\e^2),$$
and the fact that $t\,(\tan\alpha)\sqrt{d-1}=\e$ implies that the
right-hand side is $\leq 1 + \e + c\,(t^{-2}\e+\e^2)$ for some universal $c>0$. Going back to (\ref{eq:ratiolast}) we see that this finishes our upper bound for $\Pr\{\{i,j\}\in\Gamma_+\}$.\end{proof}

\subsubsection*{Correlations between edges and non-edges}

In this case we consider $s,s'\in S^{d-1}$ and look at correlations of ``edge events."

\begin{lemma}\label{lem:correlations}
For any $t\geq 1$, $0<\theta<\pi$, define
$$\xi:=1-\cos\theta,\, \gamma:=\frac{(1-\cos\,\theta)^2}{\sin\theta}.$$
Then there exists a universal constant $C>0$ such that for $s,s'\in
S^{d-1}$ such that $\inner{s}{s'}\le \cos \theta$, we have
\begin{equation}\label{eq:correlations1}\Pr\{\inner{X_{ij}}{s}\geq t,\, \inner{X_{ij}}{s'}\geq t\}\leq p\,[(C\,p\,t)^{2\xi+\xi^2} + e^{\gamma\,(1-\gamma)\,t^2 + \frac{\gamma}{1-\gamma} + \frac{\gamma^2\,t^2}{2}}\,p].\end{equation}
$$\Pr\{\inner{X_{ij}}{s}<t,\, \inner{X_{ij}}{s'}<t\}\leq 1-2p+ p\,[(C\,p\,t)^{2\,\xi+\xi^2} + e^{\gamma\,(1-\gamma)\,t^2 + \frac{\gamma}{1-\gamma} + \frac{\gamma^2\,t^2}{2}}\,p]$$\end{lemma}
\begin{remark}{\sc (nearly equal vectors.)} 
Suppose $p=o(1)$ and $\theta=o(1)$. One may check that $\gamma = (1+o(1))\,\theta^3/4$ and $\xi = (1+o(1))\,\theta^2/2$. This means that if $\theta^3\,t^2 = o(\ln(1/p))$ and 
$\theta^2\,\ln(1/t\,p)=\omega(1)$, then 
$$\Pr\{\inner{X_{ij}}{s}<t,\, \inner{X_{ij}}{s'}<t\}\leq 1-2p+o(p) = (1-p)^2\,(1+o(p)).$$
This is used in the proof of Theorem \ref{thm:connectivity}, part (iv) above.\end{remark}

\begin{proof}We focus on the inequalities in (\ref{eq:correlations1}), from which the other inqualities follow. For convenience, we write $\eta:=\cos\theta$ and note that
\begin{equation}\label{eq:gammais}\eta = 1-\xi \mbox{, so }\gamma = 1 - \frac{1 - (1+\xi)\eta}{\sqrt{1-\eta^2}}.\end{equation}
Moreover, $0<\gamma<1$: the first inequality is obvious, and the second follows from the fact that
$$0<\theta<\frac{\pi}{2}\Rightarrow 0<\gamma=\frac{(1-\cos\theta)^2}{\sin\theta}<\frac{(1-\cos\theta)\,(1+\cos\theta)}{\sin\theta} = \frac{1-\cos^2\theta}{\sin\theta}=\sin\theta<1.$$
 
Let $E$ denote the event in (\ref{eq:correlations1}). The properties of standard Gaussian vectors imply 
$$\Pr\{E\}= \Pr(\{N_1\geq t\}\cap \{\eta\,N_1 + \sqrt{1-\eta^2}\,N_2\geq t\})$$
where $N_1,N_2$ are independent standard normal random variables. 
In particular, we can upper bound
\begin{equation}\label{eq:Ecoor}\Pr\{E\} \leq \Pr\{N_1\geq (1+\xi)\,t\} + \Pr\{N_1\geq t\}\,\Pr\left\{N_2\geq \left(\frac{1-(1+\xi)\eta}{\sqrt{1-\eta^2}}\right)\,t\right\},\end{equation}

The first term in the right-hand side is $1-\Phi(t+\xi t) \leq e^{-\frac{\xi^2t}{2}-\xi t^2}\,(1-\Phi(t)) =e^{-\frac{2\xi+\xi^2}{2}\,t^2}\,(1-\Phi(t))$ by Lemma \ref{lem:tailsgaussian}. The fact that
$$\lim_{t\to +\infty}\frac{(1-\Phi(t))}{e^{-t^2/2}/(t\,\sqrt{2\pi})}=1,$$
implies that, for $t>1$, the ratio $e^{-t^2/2}/p$ is bounded by a $C\,t$, $C>0$ a constant. We conclude
\begin{equation}\label{eq:Ecoor1st}\Pr\{N_1\geq (1+\xi)\,t\}\leq p\, (e^{-t^2/2})^{2\xi+\xi^2}\leq p\,(C\,t\,p)^{2\xi+\xi^2}.\end{equation}
As for the second term in the right-hand side of (\ref{eq:Ecoor}), we apply Lemma \ref{lem:tailsgaussian} with 
$$r := \frac{t\,(1 - (1+\xi)\eta)}{\sqrt{1-\eta^2}} = (1-\gamma)\,t\mbox{ and }h := \gamma\,t.$$
We deduce:
$$\Pr\left\{N_2\geq \left(\frac{1 - (1+\xi)\eta}{\sqrt{1-\eta^2}}\right)\,t\right\}=1 - \Phi(r)\leq e^{\gamma\,(1-\gamma)\,t^2 + \frac{\gamma}{1-\gamma} + \frac{\gamma^2\,t^2}{2}}\,(1-\Phi(t)),$$
The proof finishes by combining the estimates for the right-hand side of (\ref{eq:Ecoor}).\end{proof}

\begin{lemma}\label{lem:tailsgaussian}If $\e\in(0,1/2)$, $r>0$ and $h\geq 0$,
$$e^{-h\,r - \frac{h}{r} - \frac{h^2}{2}}\leq \frac{1-\Phi(r+h)}{1-\Phi(r)}\leq e^{-h\,r - \frac{h^2}{2}}.$$\end{lemma}
\begin{proof}We first show the upper bound, namely:
\begin{equation}\label{eq:expdomcond}
 \forall r,h>0\,:\,1-\Phi(r+h)\leq e^{-r\,h - \frac{h^2}{2}}\,(1-\Phi(r)).\end{equation}
To see this, we note that:
\begin{eqnarray*}1-\Phi(r+h) &=&
  \int_{0}^{+\infty\,}\,\frac{e^{-\frac{(x+r+h)^2}{2}}}{\sqrt{2\pi}}\,dx
  \\&=& \int_{0}^{+\infty}\frac{e^{-\frac{(x+r)^2}{2}}}{\sqrt{2\pi}}{
    \,e^{-\left(x+r+\frac{h}{2}\right)\,h}}\,dx\\ 
&\leq& \int_{0}^{+\infty}\frac{e^{-\frac{(x+r)^2}{2}}}{\sqrt{2\pi}}{ \,e^{-r\,h-\frac{h^2}{2}}}\,dx\\ &=& [1-\Phi(r)]\,e^{-r\,h-\frac{h^2}{2}}.\end{eqnarray*}
To continue, we go back to the formula
$$1-\Phi(r+h)= \left(\int_{0}^{+\infty}\frac{e^{-\frac{(x+r)^2}{2}}\,e^{-(x+r)\,h}}{\sqrt{2\pi}}\,dx\right)\,e^{-\frac{h^2}{2}},$$
which is clearly related to
$$1-\Phi(r)=\int_{0}^{+\infty}\frac{e^{-\frac{(x+r)^2}{2}}}{\sqrt{2\pi}}\,dx.$$
In fact, inspection reveals that 
$$\frac{1-\Phi(r+h)}{1-\Phi(r)} = e^{-\frac{h^2}{2}}\,\E[e^{-h\,N}\mid N\geq r].$$
Using Jensen's inequality, we have
$$\frac{1-\Phi(r+h)}{1-\Phi(r)} \geq e^{-\frac{h^2}{2}}\,e^{-h\,\E[N\mid N\geq r]},$$
and (\ref{eq:expdomcond}) means that 
$\Pr\{N-r\geq t\mid N\geq r\}\leq e^{-t\,r}$, so $\E[N\mid N\geq r]\leq r + \frac{1}{r}$. We deduce:
$$\frac{1-\Phi(r+h)}{1-\Phi(r)} \geq e^{-\frac{h^2}{2}}\,e^{-h\,r - \frac{h}{r}},$$
as desired. \end{proof}

\subsection{Random graph lemmas}
\label{sec:randomgraphlemmas}

Here we collect some results on random graphs that we need in the
arguments. 
In the proof of Theorem \ref{thm:clique}
we use the following lower tail estimate of the clique number of an
Erd\H{o}s-R\'enyi random graph that follows from a standard use of
Janson's inequality.

\begin{lemma}
\label{lem:cliquenum}
Let $N_k$ denote the number of cliques of size $k$ of a $G(n,1/2-\alpha_n)$
Erd\H{o}s-R\'enyi random graph where $0\le \alpha_n \le 1/n$ and let $\delta>2$. 
Denote $\omega=2\log_2 n -2\log_2\log_2n +2\log_2 e-1$.
If $k =\lfloor \omega-\delta \rfloor$, then there exists a constant $C'$
such that for all $n$,
\[
\PROB\left\{ N_k=0 \right\} \le \exp\left( \frac{-C' n^2}{(\log_2n)^8}\right)~.
\]
\end{lemma}

\begin{proof}
Write $p=1/2-\alpha_n$ and define
$\omega_p=2\log_{1/p} n -2\log_{1/p}\log_{1/p}n +2\log_{1/p} (e/2)+1$.
We use Janson's inequality (\cite[Theorem 2.18]{JaLuRu00}) which implies
that
\[
\PROB\left\{ N_k=0 \right\} \le \exp\left( \frac{-(\EXP N_k)^2}{\Delta} \right)~,  
\]
where $\EXP N_k= \binom{n}{k}p^{\binom{k}{2}}$ and
\[
\Delta = \sum_{j=2}^k \binom{n}{k}\binom{k}{j}\binom{n-k}{k-j} 
 p^{2\binom{k-j}{2}-\binom{j}{2}-2j(k-j)}~.
\]
To bound the ratio $\Delta/(\EXP N_k)^2$, we may repeat the calculations
of Matula's theorem on the $2$-point concentration of the clique number 
(\cite{Mat72}), as in Palmer \cite[Section 5.3]{Pal85}.

Let $\beta=\log_{1/p}(3\log_{1/p} n)/\log_{1/p} n$ and define $m=\lfloor \beta k\rfloor$ Then we split the sum
\begin{eqnarray*}
\frac{\Delta}{(\EXP N_k)^2} & = &
\sum_{j=m}^k \frac{\binom{k}{j} \binom{n-k}{k-j}}{\binom{n}{k}} p^{-\binom{j}{2}} 
+ 
\sum_{j=2}^{m-1} \frac{\binom{k}{j} \binom{n-k}{k-j}}{\binom{n}{k}} p^{-\binom{j}{2}}~.
\end{eqnarray*}
To bound the first term, we write
\begin{eqnarray*}
\sum_{j=m}^k \frac{\binom{k}{j} \binom{n-k}{k-j}}{\binom{n}{k}} p^{-\binom{j}{2}}
& = &
\frac{F(m)}{\EXP N_k}~,
\end{eqnarray*}
where $F(m) = \sum_{j=m}^k \binom{k}{j} \binom{n-k}{k-j} p^{-\binom{j}{2}+\binom{k}{2}}$. Now if $k= \lfloor \omega_p-\delta \rfloor$ for some $\delta \in (0,\omega_p)$, 
then the computations in Palmer \cite[pp.77--78]{Pal85} show that
\[
F(m) \le \sum_{j=0}^\infty \left(\frac{kn\sqrt{1/p}}{p^{-k(1+\beta)/2}}\right)^j~,
\]
which is bounded whenever 
\[
\frac{kn\sqrt{(1/p)}}{p^{-k(1+\beta)/2}} =o(1)~.
\]
This is guaranteed by our choice of $\beta=\log_{1/p}(3\log_{1/p} n)/\log_{1/p} n$.
Hence, the first term is bounded by
\[
  \frac{F(m)}{\EXP N_k} = O(1) \sqrt{k}p^{k\delta/2}~.
\]
For the second term, once again just like in \cite{Pal85}, note that
\begin{eqnarray*}
\sum_{j=2}^{m-1} \frac{\binom{k}{j} \binom{n-k}{k-j}}{\binom{n}{k}} p^{-\binom{j}{2}} & \le &
O(1) \sum_{j=2}^{m-1} \frac{k^{2j}}{n^j} p^{-\binom{j}{2}} 
\\
& \le &
O(1) \sum_{j=2}^{m-1} \left(\frac{k p^{-m/2}}{n}\right)^j
\\
& \le &
O(1) \sum_{j=2}^{m-1} \left(\frac{2(\log_{1/p}n)^4}{n}\right)^j
\\
& \le &
O\left(\frac{(\log_{1/p}n)^8}{n^2}\right)~.
\end{eqnarray*}
Putting everything together, we have that there exist constants $C,C'$ such that
for $k= \lfloor \omega_p-\delta \rfloor$,
\[
\PROB\left\{ N_k=0 \right\} \le \exp\left(- C\left( \frac{(\log_{1/p}n)^8}{ n^2}
+p^{k\delta/2}\sqrt{k} \right)^{-1}\right)
\le \exp\left( \frac{-C' n^2}{(\log_2n)^8}\right)
~,  
\]
whenever $\delta >2$. Noting that $\omega_p=\omega+O(\alpha_n\log n)$
completes the proof.
\end{proof}

Part (iii) of Theorem \ref{thm:chromatic} crucially hinges on the
following interesting result of Alon and Sudakov \cite{AlSu10} on the
``resilience'' of the chromatic number of a $G(n,1/2)$ random graph.
The form of the theorem cited here does not explicitly appear in
\cite{AlSu10} but the estimates for the probability of failure follow
by a simple inspection of the proof of their Theorem 1.2.

\begin{proposition}
\label{prop:AlonSudakov}
{\sc (\cite[Theorem 1.2]{AlSu10}).}
There exist positive constants $c_1,c_2$ such that the following
holds.  Let $\epsilon >0$ and let $G$ be a $G(n,1/2)$ random graph.
With probability at least $1-\exp(c_1n^2/(\log n)^4)$, for every
collection $E$ of at most $c_2\epsilon^2n^2/(\log_2n)^2$ edges, the
chromatic number of $G\cup E$ is at most $(1+\epsilon)n/(2\log_2 n)$.
\end{proposition}

The final lemma is used in proving part (i) of Theorem
\ref{thm:connectivity}.

\begin{lemma}
\label{lem:oconnell}
Fix $c\in (0,1)$. With $p=c \log n/n$, let $N$ be the number of isolated vertices in $G(n,p)$. Then for $n$ large, 
$\PROB(N=0) \le \exp(-n^{1-c}/3)$. 
\end{lemma}
\begin{proof}
The following approach is borrowed from O'Connell \cite{oconnell}. Fix $q=1-\sqrt{1-p}$ and let $D(n,q)$ be the random directed graph 
with vertices $[n]$ in which each oriented edge $ij$ appears independently with probability $q$. 
Write $I$ for the number of vertices of $D(n,q)$ with no incoming edges, and $M$ for the number of isolated 
vertices in $D(n,q)$, with no incoming or outgoing edges. Then $M$ and $N$ have the same distribution. 
Next, observe that $I$ has law $\BIN{n}{(1-q)^{n-1}} = \BIN{n}{(1-p)^{(n-1)/2}}$. Furthermore, conditional on $I$, 
\[
M \stackrel{\mathrm{d}}{=} \BIN{I}{(1-p)^{(n-I)/2}}. 
\]
It follows that 
\begin{align}\label{eq:bin_split}
\PROB(N=0) & = \PROB(M=0) \nonumber\\
			& \le \PROB(|I-\EXP I| > \EXP I/2) + \sup_{k \in (1/2)\EXP I, (3/2)\EXP I} \PROB(\BIN{k}{(1-p)^{(n-k)/2}} = 0). 
\end{align}
For the first term, a Chernoff bound gives 
\begin{equation}\label{eq:bin_split1}
\PROB(|I-\EXP I| > \EXP I/2) \le 2e^{-\EXP I/10} = 2e^{-n(1-p)^{(n-1)/2}/10} = e^{-(1+o(1)) n^{1-c/2}/10}\, ,
\end{equation}
where the last inequality holds since $(1-p)^{(n-1)/2} =(1+o(1) n^{-c/2}$. 
Next, fix $k$ as in the above supremum. For such $k$ we have $p(n-k) = c \log n + O(\log n/n^{c/2})$. 
Using this fact and that $1-p \ge e^{-p-p^2}$ for $p$ small yields 
\begin{align*}
\PROB(\BIN{k}{(1-p)^{(n-k)/2}}=0) 	& = (1-(1-p)^{(n-k)/2})^k \\
							& \le \exp\left(-k(1-p)^{(n-k)/2}\right) \\
							& = \exp\left(-k e^{-(p+p^2)(n-k)/2}\right) \\
							& = \exp\left(-(1+o(1))kn^{-c/2} \right)\, .
\end{align*}
Using that $1-p \ge e^{-p-p^2}$ a second time gives 
\[
k \ge \EXP I/2 = n(1-p)^{(n-1)/2}/2 \ge (1+o(1))ne^{-np/2}/2 = (1+o(1))n^{1-c/2}/2.
\] 
The two preceding inequalities together imply that 
\[
\PROB(\BIN{k}{(1-p)^{(n-k)/2}}=0) \le \exp\left(-(1/2+o(1)) \cdot n^{1-c}\right)\, .
\]
Using this bound and (\ref{eq:bin_split1}) in the inequality (\ref{eq:bin_split}), the result follows easily. 
\end{proof}

\let\OLDthebibliography\thebibliography
\renewcommand\thebibliography[1]{
  \OLDthebibliography{#1}
  \setlength{\parskip}{2pt}
  \setlength{\itemsep}{2pt plus 0.3ex}
}

\end{document}